\documentclass[amssymb,12pt ]{article}
\pdfoutput=1
\usepackage{geometry}
\usepackage{amsfonts}
\usepackage{amssymb}
\usepackage{amsmath}
\usepackage{amsthm}
\usepackage{graphicx}
\usepackage{picture}
\usepackage{curves}
\usepackage{subfigure}
\usepackage{epic}
\usepackage{here}
\usepackage{color}

\theoremstyle{plain}
\newtheorem{theorem}{Theorem}[section]
\newtheorem{corollary}[theorem]{Corollary}
\newtheorem{lemma}[theorem]{Lemma}
\newtheorem{proposition}[theorem]{Proposition}

\newtheorem{exa}[theorem]{Example}

\newenvironment{ack}{\noindent{\bf Acknowledgments}}

\theoremstyle{definition}
\newtheorem{definition}[theorem]{Definition}

\newcommand{\vol}{{\rm vol}}
\newcommand{\cs}{{\rm cs}}
\newcommand{\li}{{\rm Li}_2}

\newcommand{\modulo}{~~({\rm mod}~\pi^2)}

\begin{document}

\title{Reidemeister transformations of the potential function and the solution
}
\author{\sc Jinseok Cho and Jun Murakami}
\maketitle
\begin{abstract}
The potential function of the optimistic limit of the colored Jones polynomial and the construction of the solution of the hyperbolicity equations
were defined in the authors' previous articles. In this article, we define the Reidemeister transformations of the
potential function and the solution by the changes of them under the Reidemeister moves of the link diagram
and show the explicit formulas. 
These two formulas enable us to see the changes of the complex volume
formula under the Reidemeister moves. As an application,
we can simply specify the discrete faithful representation of the link group by
showing a link diagram and one geometric solution.
\end{abstract}

\section{Introduction}\label{sec1}

\subsection{Overview}

One of the fundamental theorem of knot theory is the Reidemeister theorem, which states
two links are equivalent if and only if their diagrams are related by finite steps of the Reidemeister moves.
Therefore, one of the most natural method to obtain a knot invariant is to define a value from a knot diagram
and show that the value is invariant under the Reidemeister moves. 
However, some invariants cannot be defined in this way, especially the ones defined from the hyperbolic structure of the link.
This is because, the hyperbolicity equations\footnote{{\it Hyperbolicity equations} are the gluing equations of the edges of a given triangulation
together with the equation of the completeness condition. Each solution of the equations determines a boundary-parabolic representation 
and some of them determines the hyperbolic structure of the link.}
and their solutions do not change locally under the Reidemeister moves. 
Especially, when a solution of certain hyperbolicity equations is given,
we cannot see how the equations and the solution change under the Reidemeister moves. 
This is one of the major obstructions to develop a combinatorial approach to the hyperbolic invariants of links.

On the other hand, {the optimistic limit method was first introduced at \cite{Murakami00b}.
Although this method was not defined rigorously, the resulting value was {\it optimistically} expected to be the actual limit
of certain quantum invariants. The rigorous definition of the optimistic limit of the Kashaev invariant 
was proposed at \cite{Yokota10} and the resulting value was proved to be the complex volume of the knot. 
Although this definition is rigorous and general enough, 
they requires some unnatural assumptions on the diagram and several technical difficulties.
Therefore it was modified to more combinatorial version at \cite{Cho13a}.
The optimistic limit used in this article is the one defined at \cite{Cho13b} and the main results are based on \cite{Cho14c}.

In our definition, the triangulation is naturally defined from the link diagram
and its hyperbolicity equations, whose solutions determine the boundary-parabolic representations\footnote{
A representation $\rho:\pi_1(L)\rightarrow {\rm PSL}(2,\mathbb{C})$ of the link group $\pi_1(L):=\pi_1(\mathbb{S}^3\backslash L)$ is {\it boundary-parabolic} when any meridian loop of the boundary-tori of the link complement $\mathbb{S}^3\backslash L$ maps to a parabolic element in ${\rm PSL}(2,\mathbb{C})$ under $\rho$.} of the link group, 
are the partial derivatives 
$\exp(w_k\frac{\partial W}{\partial w_k})=1$ $(k=1,\ldots,n)$ of certain potential function $W(w_1,\ldots,w_n)$.
Note that this potential function is combinatorially defined from the link diagram, 
so it changes naturally under the Reidemeister moves.
Then the optimistic limit is defined by the evaluation of the potential function (with slight modification) at certain solution of the hyperbolicity equations. (Explicit definition is the equation (\ref{optimistic}).)

Let $\mathcal{P}$ be the conjugation quandle consisting of the parabolic elements 
of ${\rm PSL}(2,\mathbb{C})$ proposed at \cite{Kabaya14}.
The shadow-coloring of $\mathcal{P}$ is a way of assigning elements of $\mathcal{P}$ to arcs and regions of the link diagram. The elements on arcs are naturally determined from $\rho$ and the ones on the regions are from certain rules. 
According to \cite{Kabaya14} and \cite{Cho14c},
we can construct the developing map of a given boundary-parabolic representation $\rho:\pi_1(L)\rightarrow {\rm PSL}(2,\mathbb{C})$ directly from the shadow-coloring. 
(The explicit construction is in Figure \ref{fig06}. Note that this construction is based on \cite{Neumann99} and \cite{Zickert09}.)}

This construction of the solution has two major advantages. At first, if a boundary-parabolic representation $\rho$ is given, then we can always
construct the solution corresponding to $\rho$ for any link diagram. (This was the main theorem of \cite{Cho14c}.) In other words,
for the hyperbolicity equations of our triangulation, we can always guarantee the existence of a geometric solution,\footnote{
{\it Geometric solution} is a solution of the hyperbolicity equations which determines the discrete faithful representation.
(Unlike the standard definition, we allow some tetrahedron can have negative volume.
If we consider the triangulation of $\mathbb{S}^3\backslash (L\cup\{\text{two points}\}$), 
the negative volume tetrahedra are unavoidable.) Note that geometric solution in our context is not unique.} which is an assumption in many other texts.
Furthermore, the constructed solution changes locally under the Reidemeister moves on the link diagram $D$.
Note that the variables $w_1,\ldots,w_n$ of the hyperbolicity equations are assinged to regions of the diagram.
(See Section \ref{sec12} below.)
Assume the solution $(w_1^{(0)},\ldots,w_n^{(0)})$ is constructed from the diagram $D$ together with the representation $\rho$. 

{
\begin{definition} 
A solution $(w_1^{(0)},\ldots,w_n^{(0)})$ is called {\it essential} when 
$w_k^{(0)}\neq 0$ for all $k=1,\ldots,n$ and $w_k^{(0)}\neq w_m^{(0)}$ for the pairs $w_k^{(0)}$ and $w_m^{(0)}$
assigned to adjacent regions of the diagram $D$.
\end{definition}

According to Lemma \ref{lem}, essentialness of a solution is generic property, so we can always construct uncountably many essential solutions from any $D$ and $\rho$. From now on, solutions in this article are always assumed to be essential
and the Reidemeister transformations are defined between two essential solutions.
(This assumption is guaranteed by Corollary \ref{cor32}.)
Note that essentialness of the solution guarantees that the shape parameters defined in Section \ref{sec33}
are not in $\{0,1,\infty\}$.


Let $D'$ be the link diagram obtained by applying one Reidemeister move to $D$.}
In this article, we will show that if a new variable $w_{n+1}$ is appeared in $D'$, then 
the values $w_1^{(0)},\ldots,w_n^{(0)}$ of the newly constructed solution from $D'$ and $\rho$ are preserved and
the value $w_{n+1}^{(0)}$ is uniquely determined by the other values $w_1^{(0)},\ldots,w_n^{(0)}$. 
(The explicit relations are in Section \ref{sec4}.)
Also, if a region with $w_k$ is removed in $D'$, then we can easily get the solution by removing the value $w_k^{(0)}$ of the variable $w_k$.
These changes of the solution will be called {\it the Reidemeister transformations of the solution} in Section \ref{sec12}.

Using the Reidemeister transformations of the potential function together with the solution,
we can see how the complex volume formula changes under the Reidemeister moves.
(See Theorem \ref{thm1}.)
As an application, we can easily specify the discrete faithful representation by showing one link diagram $D$ and one geometric solution corresponding to the diagram. 
In particular, if we have another diagram $D'$ of $L$, then we can easily find the geometric solution corresponding to $D'$,
{without solving the hyperbolicity equations again,} by applying the Reidemeister transformations of the solution.

Many results of the optimistic limit and other concepts used in this article are scattered in the authors' previous articles.
Referring all of them might be quite confusing for readers,
so we added many known results here, especially in Sections \ref{sec2}-\ref{sec3}, 
and sometimes we reprove the known results to clarify the discussion.

\subsection{Reidemeister transformations}\label{sec12}

To describe the exact definition of the Reidemeister transformation, we have to define the potential function first.
Consider a link diagram\footnote{
We always assume the diagram $D$ does not contain a trivial knot component which has only over-crossings or under-crossings or no crossing.
If it happens, we change the diagram of the trivial component slightly by adding a kink.
} $D$ of a link $L$ and assign complex variables $w_1,\ldots,w_n$ to regions of $D$. Then we define the potential function of a crossing $j$ as in Figure \ref{pic01}.

\begin{figure}[h]
\setlength{\unitlength}{0.4cm}
\subfigure[Positive crossing]{
  \begin{picture}(35,6)\thicklines
    \put(6,5){\vector(-1,-1){4}}
    \put(2,5){\line(1,-1){1.8}}
    \put(4.2,2.8){\vector(1,-1){1.8}}
    \put(3.5,1){$w_a$}
    \put(5.5,3){$w_b$}
    \put(3.5,4.5){$w_c$}
    \put(1.5,3){$w_d$}
    \put(8,3){$\longrightarrow$}
    \put(11,4){$W^j:=-\li(\frac{w_c}{w_b})-\li(\frac{w_c}{w_d})+\li(\frac{w_a w_c}{w_b w_d})+\li(\frac{w_b}{w_a})+\li(\frac{w_d}{w_a})$}
    \put(15,2){$-\frac{\pi^2}{6}+\log\frac{w_b}{w_a}\log\frac{w_d}{w_a}$}
    \put(3.6,2){$j$}
  \end{picture}}\\
\subfigure[Negative crossing]{
  \begin{picture}(35,6)\thicklines
    \put(2,5){\vector(1,-1){4}}
    \put(6,5){\line(-1,-1){1.8}}
    \put(3.8,2.8){\vector(-1,-1){1.8}}
    \put(3.5,1){$w_a$}
    \put(5.5,3){$w_b$}
    \put(3.5,4.5){$w_c$}
    \put(1.5,3){$w_d$}
    \put(8,3){$\longrightarrow$}
    \put(11,4){$W^j:=\li(\frac{w_c}{w_b})+\li(\frac{w_c}{w_d})-\li(\frac{w_a w_c}{w_b w_d})-\li(\frac{w_b}{w_a})-\li(\frac{w_d}{w_a})$}
    \put(15,2){$+\frac{\pi^2}{6}-\log\frac{w_b}{w_a}\log\frac{w_d}{w_a}$}
        \put(3.6,2){$j$}
  \end{picture}}
  \caption{Potential function of the crossing $j$}\label{pic01}
\end{figure}

{In the definition above, $\li(z):=-\int_0^z \frac{\log(1-t)}{t}dt$ is the dilogarithm function. 
Although it is a multi-valued function depending on the choice of the branches of $\log t$ and $\log(1-t)$,
the final formula in (\ref{optimistic}) does not depend on choice of the branches. (See Lemma 2.1's of \cite{Cho13a} and \cite{Cho13c}).}

{\it The potential function} of $D$ is defined by
\begin{equation}\label{defW}
W(w_1,\ldots,w_n):=\sum_{j\text{ : crossings of }D}W^j,
\end{equation}
and we modify it to
\begin{equation}\label{defW_0}
W_0(w_1,\ldots,w_n):=W(w_1,\ldots,w_n)-\sum_{k=1}^n \left(w_k\frac{\partial W}{\partial w_k}\right)\log w_k.
\end{equation}

Also, we define the set of equations
\begin{equation}\label{defH}
\mathcal{I}:=\left\{\left.\exp\left(w_k\frac{\partial W}{\partial w_k}\right)=1\right|k=1,\ldots,n\right\}.
\end{equation}
Then $\mathcal{I}$ becomes the set of the hyperbolicity equations of the five-term triangulation
defined in Section \ref{sec2}. (See Proposition \ref{pro21}.)

Consider a boundary-parabolic representation $\rho:\pi_1(L)\rightarrow{\rm PSL}(2,\mathbb{C})$. Then, using the shadow-coloring of
$\mathcal{P}$ induced by $\rho$, we can construct the solution $\bold w^{(0)}=(w_1^{(0)},\ldots,w_n^{(0)})$ of $\mathcal{I}$ satisfying
$\rho_{\bold w^{(0)}}=\rho$ up to conjugation, where $\rho_{\bold w^{(0)}}$ is the representation induced 
by the five-term triangulation together with the solution $\bold w^{(0)}$. (The detail is in Section \ref{sec3}.
See Proposition \ref{pro2}. {We will also show that any solution of $\mathcal{I}$ can be constructed by this method in Appendix \ref{app}.}) Furthermore, the solution satisfies
\begin{equation}\label{optimistic}
W_0(\bold w^{(0)})\equiv i(\vol(\rho)+i\,\cs(\rho))~~({\rm mod}~\pi^2),
\end{equation}
where $\vol(\rho)$ and $\cs(\rho)$ are the hyperbolic volume and the Chern-Simons invariant of $\rho$, respectively, 
{which were defined in \cite{Neumann04} and \cite{Zickert09}.}
We call $\vol(\rho)+i\,\cs(\rho)$ {\it the (hyperbolic) complex volume of $\rho$}
and define {\it the optimistic limit of the colored Jones polynomial} by $W_0(\bold w^{(0)})$.

The oriented Reidemeister moves defined in \cite{Polyak10} are in Figure \ref{pic02}. 
We define the potential functions $V_{R1}$, $V_{R1'}$, $V_{R_2}$, $V_{R2'}$, $V_{R3}$ and $V_{R3'}$ of Figure \ref{pic02} as follows:
\begin{align*}
V_{R1}(w_a,w_b,w_c)=&\li(\frac{w_b}{w_a})+\li(\frac{w_b}{w_c})-\li(\frac{w_b^2}{w_a w_c})\\
  &-\li(\frac{w_a}{w_b})-\li(\frac{w_c}{w_b})+\frac{\pi^2}{6}-\log\frac{w_a}{w_b}\log\frac{w_c}{w_b},
  \end{align*}
\begin{align*}
V_{R1'}(w_a,w_b,w_c)=&-\li(\frac{w_b}{w_a})-\li(\frac{w_b}{w_c})+\li(\frac{w_b^2}{w_a w_c})\\
  &+\li(\frac{w_a}{w_b})+\li(\frac{w_c}{w_b})-\frac{\pi^2}{6}+\log\frac{w_a}{w_b}\log\frac{w_c}{w_b},
\end{align*}
\begin{align*}
V_{R2}&(w_a,w_b,w_c,w_d,w_e)=\\
  &\li(\frac{w_a}{w_d})+\li(\frac{w_a}{w_e})-\li(\frac{w_a w_c}{w_d w_e})-\li(\frac{w_d}{w_c})-\li(\frac{w_e}{w_c})-\log\frac{w_d}{w_c}\log\frac{w_e}{w_c}\\
  &-\li(\frac{w_c}{w_b})-\li(\frac{w_c}{w_e})+\li(\frac{w_a w_c}{w_b w_e})+\li(\frac{w_b}{w_a})+\li(\frac{w_e}{w_a})+\log\frac{w_b}{w_a}\log\frac{w_e}{w_a},
  \end{align*}
\begin{align*}
V_{R2'}&(w_a,w_b,w_c,w_d,w_e)=\\
  &\li(\frac{w_c}{w_d})+\li(\frac{w_c}{w_e})-\li(\frac{w_a w_c}{w_d w_e})-\li(\frac{w_d}{w_a})-\li(\frac{w_e}{w_a})-\log\frac{w_d}{w_a}\log\frac{w_e}{w_a}\\
  &-\li(\frac{w_a}{w_b})-\li(\frac{w_a}{w_e})+\li(\frac{w_a w_c}{w_b w_e})+\li(\frac{w_b}{w_c})+\li(\frac{w_e}{w_c})+\log\frac{w_b}{w_c}\log\frac{w_e}{w_c},
\end{align*}
\begin{align*}
V_{R3}&(w_a,w_b,w_c,w_d,w_e,w_f,w_h)=-\frac{\pi^2}{2}\\
  &-\li(\frac{w_e}{w_d})-\li(\frac{w_e}{w_f})+\li(\frac{w_e w_h}{w_d w_f})+\li(\frac{w_d}{w_h})+\li(\frac{w_f}{w_h})+\log\frac{w_d}{w_h}\log\frac{w_f}{w_h}\\
  &-\li(\frac{w_f}{w_a})-\li(\frac{w_f}{w_h})+\li(\frac{w_b w_f}{w_a w_h})+\li(\frac{w_a}{w_b})+\li(\frac{w_h}{w_b})+\log\frac{w_a}{w_b}\log\frac{w_h}{w_b}\\
  &-\li(\frac{w_h}{w_b})-\li(\frac{w_h}{w_d})+\li(\frac{w_c w_h}{w_b w_d})+\li(\frac{w_b}{w_c})+\li(\frac{w_d}{w_c})+\log\frac{w_b}{w_c}\log\frac{w_d}{w_c},
  \end{align*}
\begin{align*}
V_{(R3)^{-1}}&(w_a,w_b,w_c,w_d,w_e,w_f,w_g)=-\frac{\pi^2}{2}\\
  &-\li(\frac{w_f}{w_a})-\li(\frac{w_f}{w_e})+\li(\frac{w_g w_f}{w_a w_e})+\li(\frac{w_a}{w_g})+\li(\frac{w_e}{w_g})+\log\frac{w_a}{w_g}\log\frac{w_e}{w_g}\\
  &-\li(\frac{w_e}{w_d})-\li(\frac{w_e}{w_g})+\li(\frac{w_c w_e}{w_d w_g})+\li(\frac{w_d}{w_c})+\li(\frac{w_g}{w_c})+\log\frac{w_d}{w_c}\log\frac{w_g}{w_c}\\
  &-\li(\frac{w_g}{w_a})-\li(\frac{w_g}{w_c})+\li(\frac{w_b w_g}{w_a w_c})+\li(\frac{w_a}{w_b})+\li(\frac{w_c}{w_b})+\log\frac{w_a}{w_b}\log\frac{w_c}{w_b}.
\end{align*}
Note that $V_{R1}$ is the potential function of the diagram obtained after applying $R1$ move in Figure \ref{pic02}(a). 
All others are defined in the same ways,
for example, $V_{R3}$ and $V_{R3'}$ are the potential functions of the right-hand and the left-hand sides of Figure \ref{pic02}(c), respectively.

\begin{figure}
\centering
\subfigure[First moves]{\includegraphics[scale=1.2]{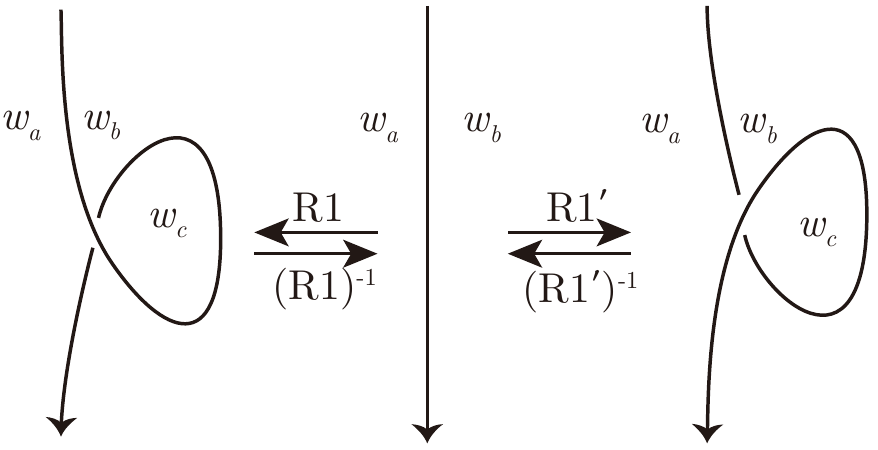}}\\
\subfigure[Second moves]{\includegraphics[scale=1.2]{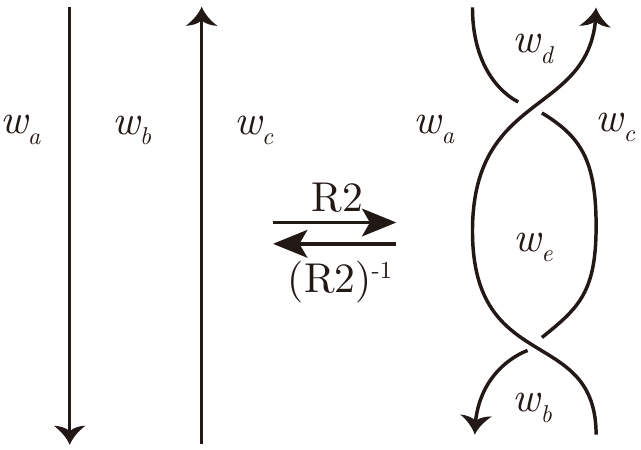}\hspace{0.5cm}\includegraphics[scale=1.2]{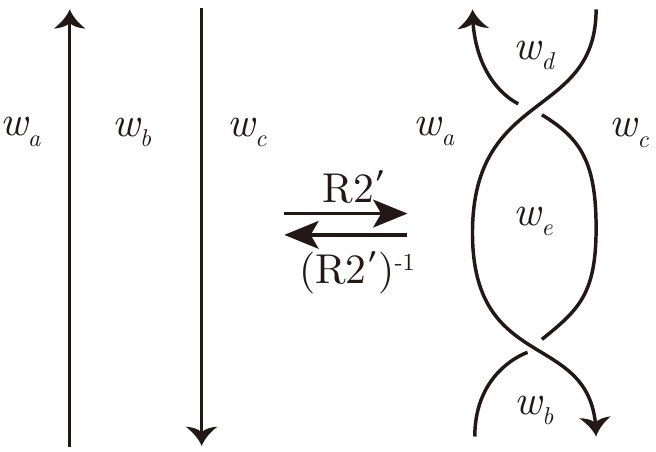}}\\
\subfigure[Third move]{\includegraphics[scale=1.2]{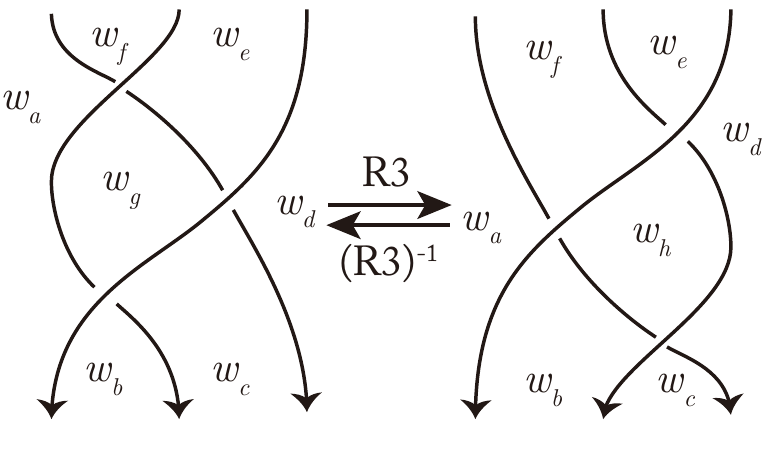}}
 \caption{Oriented Reidemeister moves}\label{pic02}
\end{figure}

\begin{definition}\label{def1}
{\it The Reidemeister transformations}\label{def1} $T_{R1}, T_{(R1)^{-1}},...,T_{R3}, T_{(R3)^{-1}}$ of the potential function $W(\ldots,w_a,w_b,\ldots)$ are defined as follows:
\begin{align}
T_{R1}(W)(\ldots,w_a,w_b,w_c,\ldots)&=W+V_{R1},\nonumber\\
T_{(R1)^{-1}}(W)(\ldots,w_a,w_b,\ldots)&=W-V_{R1},\nonumber\\
T_{R1'}(W)(\ldots,w_a,w_b,w_c,\ldots)&=W+V_{R1'},\nonumber\\
T_{(R1')^{-1}}(W)(\ldots,w_a,w_b,\ldots)&=W-V_{R1'},\nonumber\\
T_{R2}(W)(\ldots,w_a,w_b,w_c,w_d,w_e,\ldots)&=W+V_{R2},\label{R2a}\\
T_{(R2)^{-1}}(W)(\ldots,w_a,w_b,w_c,\ldots)&=W-V_{R2},\label{R2b}\\
T_{R2'}(W)(\ldots,w_a,w_b,w_c,w_d,w_e,\ldots)&=W+V_{R2'},\label{R2c}\\
T_{(R2')^{-1}}(W)(\ldots,w_a,w_b,w_c,\ldots)&=W-V_{R2'},\label{R2d}\\
T_{R3}(W)(\ldots,w_a,w_b,w_c,w_d,w_e,w_f,w_h,\ldots)&=W+V_{R3}-V_{(R3)^{-1}},\nonumber\\
T_{(R3)^{-1}}(W)(\ldots,w_a,w_b,w_c,w_d,w_e,w_f,w_g,\ldots)&=W-V_{R3}+V_{(R3)^{-1}}.\nonumber
\end{align}
Note that, when applying $R2$ (or $R2'$) move in (\ref{R2a}) (or (\ref{R2c})), we replace $w_b$ of $W$ with $w_d$ for the potential functions of the crossings adjacent to the region associated with $w_d$.
Also, when applying $(R2)^{-1}$ (or $(R2')^{-1}$) move in (\ref{R2b}) (or (\ref{R2d})), we replace $w_d$ of $W$ with $w_b$.
\end{definition}

Remark that the Reidemeister transformations of the potential function is nothing but the changes of the potential function defined in (\ref{defW})
under the corresponding Reidemeister moves.

%
%

\begin{definition}
{\it The Reidemeister transformations}\label{def1} $T_{R1}, T_{(R1)^{-1}},...,T_{R3}, T_{(R3)^{-1}}$ of the {essential} solution
$(\ldots,w_a^{(0)},w_b^{(0)},\ldots)$ of $\mathcal{I}$ in (\ref{defH}) is defined as follows: for the first Reidemeister moves
\begin{align*}
T_{R1}(\ldots,w_a^{(0)},w_b^{(0)},\ldots)=T_{R1'}(\ldots,w_a^{(0)},w_b^{(0)},\ldots)
=(\ldots,w_a^{(0)},w_b^{(0)},w_c^{(0)},\ldots),
\end{align*}
where $w_c^{(0)}=2w_b^{(0)}-w_a^{(0)}$, and
\begin{align*}
T_{(R1)^{-1}}(\ldots,w_a^{(0)},w_b^{(0)},w_c^{(0)},\ldots)=T_{(R1')^{-1}}(\ldots,w_a^{(0)},w_b^{(0)},w_c^{(0)},\ldots)
=(\ldots,w_a^{(0)},w_b^{(0)},\ldots).
\end{align*}

For the second Reidemeister moves, we put $T_{R2}(W)$ (or $T_{R2'}(W)$) be the potential function 
in (\ref{R2a}) (or (\ref{R2c})). Then{
\begin{align*}
T_{R2}(\ldots,w_a^{(0)},w_b^{(0)},w_c^{(0)},\ldots)=(\ldots,w_a^{(0)},w_b^{(0)},w_c^{(0)},w_d^{(0)},w_e^{(0)},\ldots)\\
\left(\text{or }T_{R2'}(\ldots,w_a^{(0)},w_b^{(0)},w_c^{(0)},\ldots)=(\ldots,w_a^{(0)},w_b^{(0)},w_c^{(0)},w_d^{(0)},w_e^{(0)},\ldots)\right),
\end{align*}
where $w_d^{(0)}=w_b^{(0)}$ and $w_e^{(0)}$ is uniquely determined by the equation 
\begin{equation}\label{TR2}
\exp(w_b\frac{\partial T_{R2}(W)}{\partial w_b})=1 \left(\text{or } \exp(w_b\frac{\partial T_{R2'}(W)}{\partial w_b})=1\right),
\end{equation}
and
\begin{align*}
T_{(R2)^{-1}}(\ldots,w_a^{(0)},w_b^{(0)},w_c^{(0)},w_d^{(0)},w_e^{(0)},\ldots)=(\ldots,w_a^{(0)},w_b^{(0)},w_c^{(0)},\ldots)\\
\left(\text{or } 
T_{(R2')^{-1}}(\ldots,w_a^{(0)},w_b^{(0)},w_c^{(0)},w_d^{(0)},w_e^{(0)},\ldots)=(\ldots,w_a^{(0)},w_b^{(0)},w_c^{(0)},\ldots)\right).
\end{align*}
Note that the equation (\ref{TR2}) can be expressed explicitly
by using the parameters around the region of $w_b$. 
(Explicit expression of (\ref{TR2}) is in Lemma \ref{lem53}.)}

For the third Reidemeister moves,
\begin{align*}
&T_{R3}(\ldots,w_a^{(0)},w_b^{(0)},w_c^{(0)},w_d^{(0)},w_e^{(0)},w_f^{(0)},w_g^{(0)},\ldots)\\
&~~~=(\ldots,w_a^{(0)},w_b^{(0)},w_c^{(0)},w_d^{(0)},w_e^{(0)},w_f^{(0)},w_h^{(0)},\ldots),\\
&T_{(R3)^{-1}}(\ldots,w_a^{(0)},w_b^{(0)},w_c^{(0)},w_d^{(0)},w_e^{(0)},w_f^{(0)},w_h^{(0)},\ldots)\\
&~~~=(\ldots,w_a^{(0)},w_b^{(0)},w_c^{(0)},w_d^{(0)},w_e^{(0)},w_f^{(0)},w_g^{(0)},\ldots),
\end{align*}
where $w_h^{(0)}$ or $w_g^{(0)}$ is uniquely determined by the equation 
\begin{equation}\label{R3rel}
w_d^{(0)} w_g^{(0)}-w_c^{(0)} w_e^{(0)}=w_a^{(0)} w_h^{(0)} -w_b^{(0)} w_f^{(0)}.
\end{equation}

\end{definition}

\begin{theorem}\label{thm1} 
For a link diagram $D$ of $L$, let $\bold w=(\ldots,w_a,w_b,\ldots)$ and put $W(\bold w)$ be the potential function of $D$. 
Consider a boundary-parabolic representation $\rho:\pi_1(L)\rightarrow{\rm PSL}(2,\mathbb{C})$
and let $\bold w^{(0)}$ be the solution of $\mathcal{I}$ constructed by the shadow-coloring of $\mathcal{P}$ 
satisfying $\rho_{\bold w^{(0)}}=\rho$, up to conjugation. (See Proposition \ref{pro2} for the actual construction of $\bold w^{(0)}$.)
Then, for any Reidemeister transformation $T$, $T(\bold w^{(0)})$ is also a solution\footnote{
The resulting solution of a Reidemeister transformation on an essential solution can be nonessential.
However, essential solutions are generic, so we can deform both solutions to essential ones by changing the region-colorings slightly. (See Corollary \ref{cor32}.)
Therefore, we assume the solutions are always essential.}
of $\mathcal{I}$ and the induced representation $\rho_{T(\bold w^{(0)})}$ satisfies $\rho_{T(\bold w^{(0)})}=\rho$, up to conjugation.
Furthermore,
\begin{equation}\label{volume}
T(W)_0(T(\bold w^{(0)}))\equiv W_0(\bold w^{(0)}))\equiv i(\vol(\rho)+i\,\cs(\rho))~~({\rm mod}~\pi^2),
\end{equation}
where $T(W)_0$ is the modification of the potential function $T(W)$ by (\ref{defW_0}).
\end{theorem}

\begin{proof}
This follows from Lemma \ref{LemR1}, Lemma \ref{LemR2}, Lemma \ref{LemR3} and Proposition \ref{pro2}.
\end{proof}

Note that when some orientations of the strings in Figure \ref{pic02} are reversed,
the Reidemeister transformations of the solutions can be defined by the exactly same formula.
(It will be proved in Section \ref{sec5}.)
If we change the potential function according to the changes of the orientation, then Theorem \ref{thm1} still works. 
Therefore, it defines the {\it un-oriented} Reidemeister transformations\footnote{
The Reidemeister triansformations of the potential function still depend on the orientation.
As a matter of fact, it is possible to define the potential function of the un-oriented diagram using Section 3.2 of \cite{Cho13b}.
However, the formula will be redundantly complicate than the one defined in this article, so we do not introduce it.} of the solutions. We will discuss and prove the un-oriented ones in Section \ref{sec5}.
Also, the mirror images of the Reidemeister moves will be discussed in Section \ref{sec5}.

As an example of the Reidemeister transformations, we will show the changes of the geometric solution
of a diagram $D$ of the figure-eight knot to its mirror image $\overline{D}$ in Section \ref{sec6}.

\section{Five-term triangulation of $\mathbb{S}^3\backslash (L\cup\{\text{two points}\})$}\label{sec2}

In this section, we describe the five-term triangulation of $\mathbb{S}^3\backslash (L\cup\{\text{two points}\})$.
Many parts of explanation come from \cite{Cho13c}.

We place an octahedron ${\rm A}_j{\rm B}_j{\rm C}_j{\rm D}_j{\rm E}_j{\rm F}_j$ 
on each crossing $j$ of the link diagram as in Figure \ref{twistocta} 
so that the vertices ${\rm A}_j$ and ${\rm C}_j$ lie on the over-bridge and
the vertices ${\rm B}_j$ and ${\rm D}_j$ on the under-bridge of the diagram, respectively. 
Then we twist the octahedron by gluing edges ${\rm B}_j{\rm F}_j$ to ${\rm D}_j{\rm F}_j$ and
${\rm A}_j{\rm E}_j$ to ${\rm C}_j{\rm E}_j$, respectively. The edges ${\rm A}_j{\rm B}_j$, ${\rm B}_j{\rm C}_j$,
${\rm C}_j{\rm D}_j$ and ${\rm D}_j{\rm A}_j$ are called {\it horizontal edges} and we sometimes express these edges
in the diagram as arcs around the crossing in the left-hand side of Figure \ref{twistocta}.

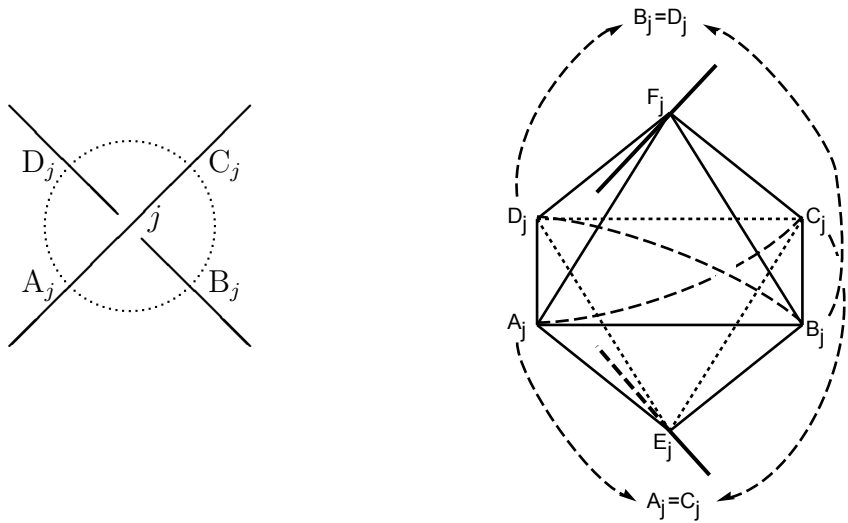
\begin{figure}[h]
\centering
\begin{picture}(6,8)  
  \setlength{\unitlength}{0.8cm}\thicklines
        \put(4,5){\arc[5](1,1){360}}
    \put(6,7){\line(-1,-1){4}}
    \put(2,7){\line(1,-1){1.8}}
    \put(4.2,4.8){\line(1,-1){1.8}}
    \put(2.2,3.9){${\rm A}_j$}
    \put(5.3,3.9){${\rm B}_j$}
    \put(5.3,5.9){${\rm C}_j$}
    \put(2.2,5.9){${\rm D}_j$}
    \put(4.3,5){$j$}
  \end{picture}\hspace{2cm}
\includegraphics[scale=0.5]{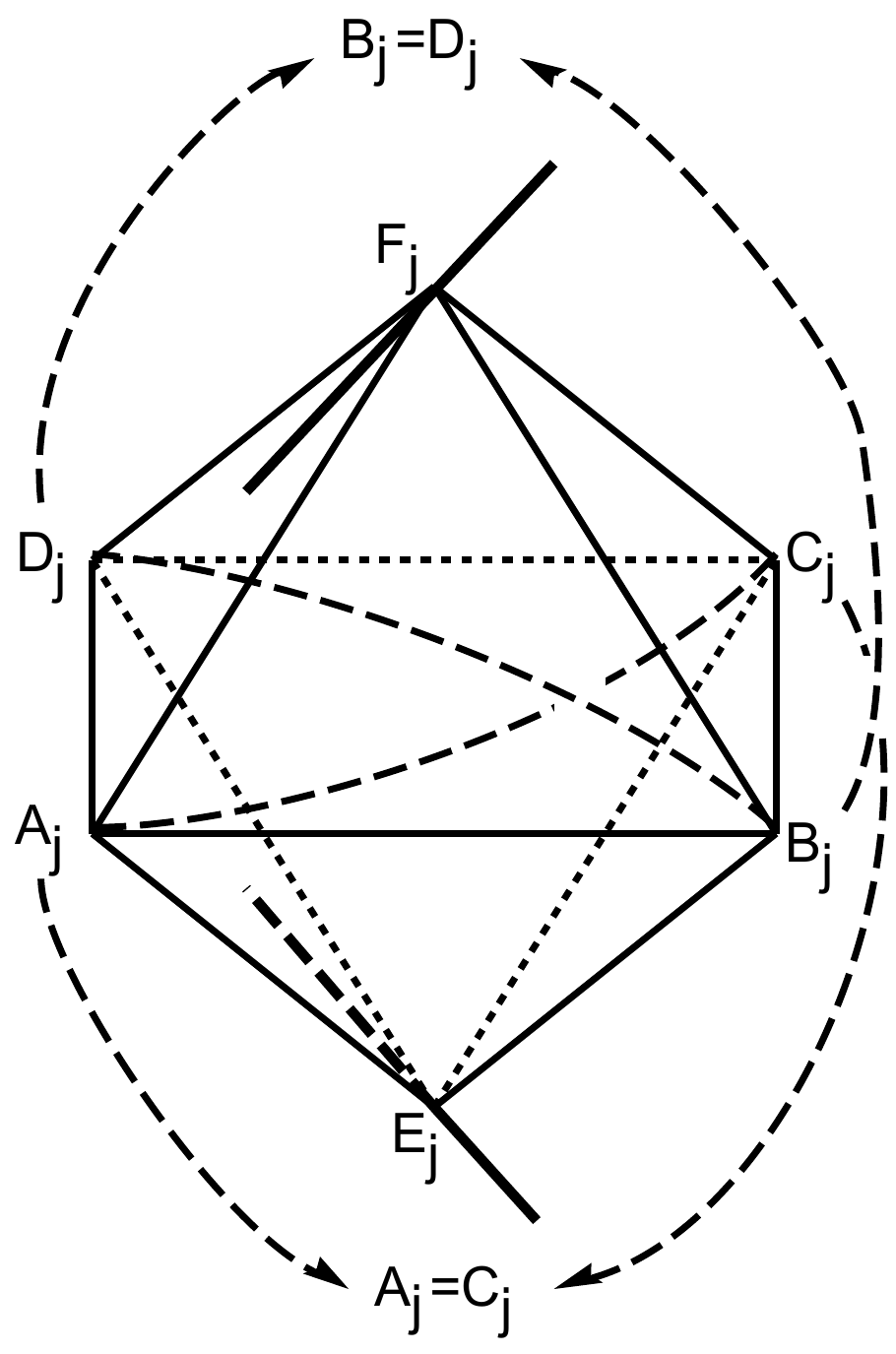}
 \caption{Octahedron on the crossing $j$}\label{twistocta}
\end{figure}

Then we glue faces of the octahedra following the edges of the link diagram. 
Specifically, there are three gluing patterns as in Figure \ref{glue pattern}.
In each cases (a), (b) and (c), we identify the faces
$\triangle{\rm A}_{j}{\rm B}_{j}{\rm E}_{j}\cup\triangle{\rm C}_{j}{\rm B}_{j}{\rm E}_{j}$ to
$\triangle{\rm C}_{j+1}{\rm D}_{j+1}{\rm F}_{j+1}\cup\triangle{\rm C}_{j+1}{\rm B}_{j+1}{\rm F}_{j+1}$,
$\triangle{\rm B}_{j}{\rm C}_{j}{\rm F}_{j}\cup\triangle{\rm D}_{j}{\rm C}_{j}{\rm F}_{j}$ to
$\triangle{\rm D}_{j+1}{\rm C}_{j+1}{\rm F}_{j+1}\cup\triangle{\rm B}_{j+1}{\rm C}_{j+1}{\rm F}_{j+1}$
and
$\triangle{\rm A}_{j}{\rm B}_{j}{\rm E}_{j}\cup\triangle{\rm C}_{j}{\rm B}_{j}{\rm E}_{j}$ to
$\triangle{\rm C}_{j+1}{\rm B}_{j+1}{\rm E}_{j+1}\cup\triangle{\rm A}_{j+1}{\rm B}_{j+1}{\rm E}_{j+1}$,
respectively. 

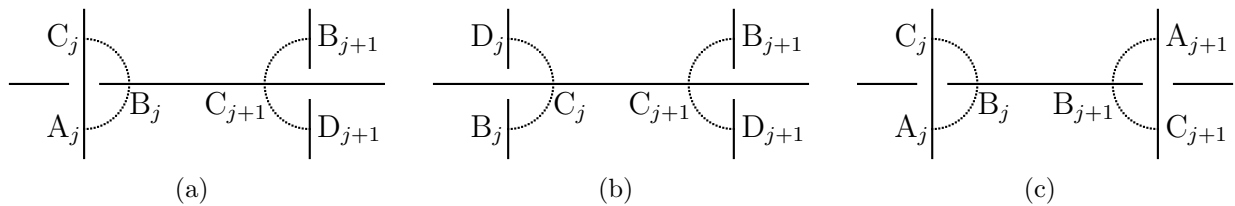
\begin{figure}[ht]
\centering
  \subfigure[]
  {\begin{picture}(5,2)\thicklines
   \put(1,1){\arc[5](0,-0.6){180}}
   \put(4,1){\arc[5](0,0.6){180}}
   \put(1.2,1){\line(1,0){3.8}}
   \put(1,2){\line(0,-1){2}}
   \put(4,0){\line(0,1){0.8}}
   \put(4,2){\line(0,-1){0.8}}
   \put(0.8,1){\line(-1,0){0.8}}
   \put(0.5,0.3){${\rm A}_j$}
   \put(1.6,0.6){${\rm B}_j$}
   \put(0.5,1.5){${\rm C}_j$}
   \put(4.1,0.3){${\rm D}_{j+1}$}
   \put(2.6,0.6){${\rm C}_{j+1}$}
   \put(4.1,1.5){${\rm B}_{j+1}$}
  \end{picture}}\hspace{0.5cm}
  \subfigure[]
  {\begin{picture}(5,2)\thicklines
   \put(1,1){\arc[5](0,-0.6){180}}
   \put(4,1){\arc[5](0,0.6){180}}
   \put(5,1){\line(-1,0){5}}
   \put(1,2){\line(0,-1){0.8}}
   \put(1,0){\line(0,1){0.8}}
   \put(4,2){\line(0,-1){0.8}}
   \put(4,0){\line(0,1){0.8}}
   \put(0.5,0.3){${\rm B}_j$}
   \put(1.6,0.6){${\rm C}_j$}
   \put(0.5,1.5){${\rm D}_j$}
   \put(4.1,0.3){${\rm D}_{j+1}$}
   \put(2.6,0.6){${\rm C}_{j+1}$}
   \put(4.1,1.5){${\rm B}_{j+1}$}
  \end{picture}}\hspace{0.5cm}
  \subfigure[]
  {\begin{picture}(5,2)\thicklines
   \put(1,1){\arc[5](0,-0.6){180}}
   \put(4,1){\arc[5](0,0.6){180}}
   \put(4.2,1){\line(1,0){0.8}}
   \put(1,2){\line(0,-1){2}}
   \put(0.8,1){\line(-1,0){0.8}}
   \put(4,2){\line(0,-1){2}}
   \put(1.2,1){\line(1,0){2.6}}
   \put(0.5,0.3){${\rm A}_j$}
   \put(1.6,0.6){${\rm B}_j$}
   \put(0.5,1.5){${\rm C}_j$}
   \put(4.1,0.3){${\rm C}_{j+1}$}
   \put(2.6,0.6){${\rm B}_{j+1}$}
   \put(4.1,1.5){${\rm A}_{j+1}$}
  \end{picture}}
  \caption{Three gluing patterns}\label{glue pattern}
\end{figure}

Note that this gluing process identifies vertices $\{{\rm A}_j, {\rm C}_j\}$ to one point, denoted by $-\infty$,
and $\{{\rm B}_j, {\rm D}_j\}$ to another point, denoted by $\infty$, and finally $\{{\rm E}_j, {\rm F}_j\}$ to
the other points, denoted by ${\rm P}_k$ where $k=1,\ldots,s$ and $s$ is the number of the components of the link $L$. 
The regular neighborhoods of $-\infty$ and $\infty$ are 3-balls and that of $\cup_{k=1}^s P_k$ is
{cone over the tori of the link $L$.}
Therefore, if we remove the vertices ${\rm P}_1,\ldots,{\rm P}_s$ from the octahedra, 
then we obtain a decomposition of $\mathbb{S}^3\backslash L$, denoted by $T$.
On the other hand, if we remove all the vertices of the octahedra, 
the result becomes an ideal decomposition of $\mathbb{S}^3\backslash (L\cup\{\pm\infty\})$.
We call the latter {\it the octahedral decomposition} and denote it by $T'$.

To obtain an ideal triangulation from $T'$, we divide each octahedron 
${\rm A}_j{\rm B}_j{\rm C}_j{\rm D}_j{\rm E}_j{\rm F}_j$ in Figure \ref{twistocta} into five ideal tetrahedra
${\rm A}_j{\rm B}_j{\rm D}_j{\rm F}_j$, ${\rm B}_j{\rm C}_j{\rm D}_j{\rm F}_j$,
${\rm A}_j{\rm B}_j{\rm C}_j{\rm D}_j$, ${\rm A}_j{\rm B}_j{\rm C}_j{\rm E}_j$
and ${\rm A}_j{\rm C}_j{\rm D}_j{\rm E}_j$.
We call the result {\it the five-term triangulation} of $\mathbb{S}^3\backslash (L\cup\{\pm\infty\})$.

Note that if we assign the shape parameter $u\in\mathbb{C}\backslash\{0,1\}$ to an edge of an ideal hyperbolic tetrahedron,
then the other edges are also parametrized by $u, u':=\frac{1}{1-u}$ and $u'':=1-\frac{1}{u}$
as in Figure \ref{fig6}.

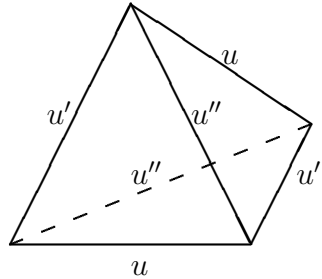
\begin{figure}[H]
\begin{center}
  {\setlength{\unitlength}{0.4cm}
  \begin{picture}(12,10)\thicklines
   \put(1,1){\line(1,0){8}}
   \put(1,1){\line(1,2){4}}
   \put(5,9){\line(1,-2){4}}
   \put(9,1){\line(1,2){2}}
   \put(5,9){\line(3,-2){6}}
   \dashline{0.5}(1,1)(11,5)
   \put(5,0){$u$}
   \put(8,7){$u$}
   \put(2.2,5){$u'$}
   \put(10.5,3){$u'$}
   \put(5,3){$u''$}
   \put(7,5){$u''$}
  \end{picture}}
  \caption{Parametrization of an ideal tetrahedron with a shape parameter $u$}\label{fig6}
\end{center}
\end{figure}
To determine the shape of the octahedron in Figure \ref{twistocta}, 
we assign shape parameters to edges of tetrahedra as in Figure \ref{fig7}. 
Note that $\frac{w_a w_c}{w_b w_d}$ in Figure \ref{fig7}(a) and $\frac{w_b w_d}{w_a w_c}$ in Figure \ref{fig7}(b)
are the shape parameters of the tetrahedron ${\rm A}_j{\rm B}_j{\rm C}_j{\rm D}_j$ assigned to the edges 
${\rm B}_j{\rm D}_j$ and ${\rm A}_j{\rm C}_j$. Also note that the assignment of shape parameters here does not
depend on the orientations of the link diagram.

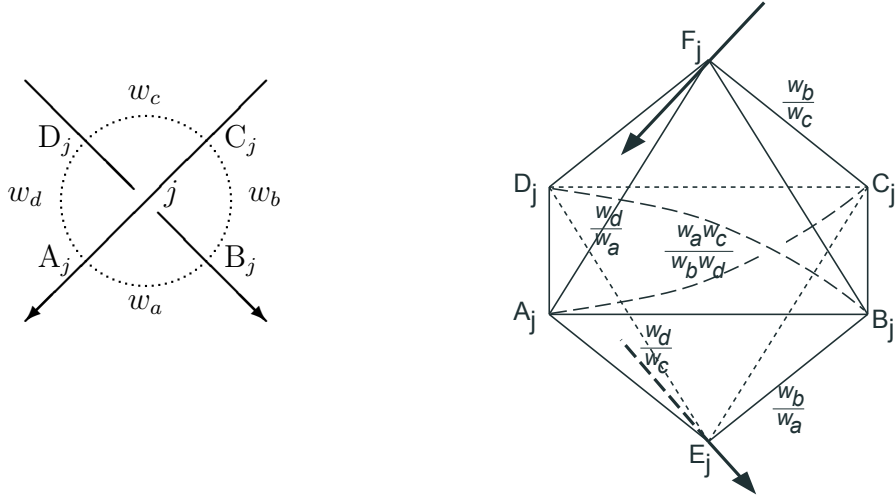
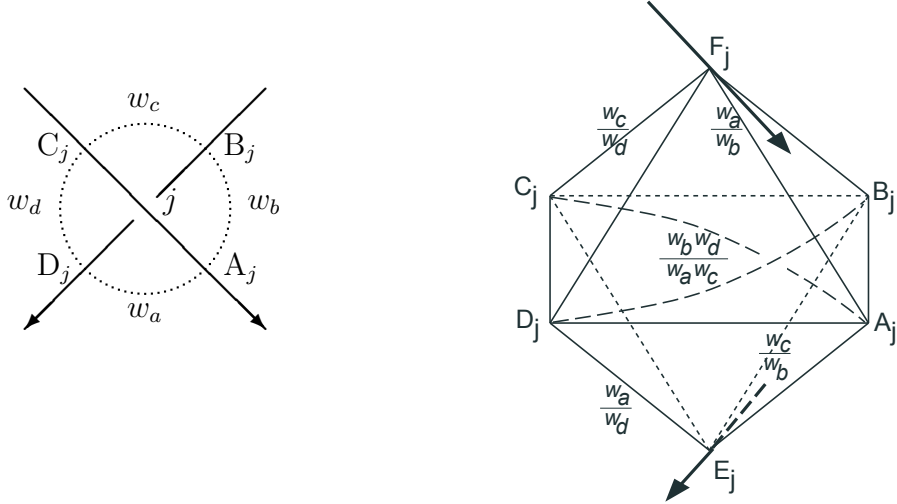
\begin{figure}[H]
\centering
  \subfigure[Positive crossing $j$]
  {\begin{picture}(6,7)  
  \setlength{\unitlength}{0.8cm}\thicklines
        \put(4,5){\arc[5](1,1){360}}
    \put(6,7){\vector(-1,-1){4}}
    \put(2,7){\line(1,-1){1.8}}
    \put(4.2,4.8){\vector(1,-1){1.8}}
    \put(3.7,3.2){$w_a$}
    \put(5.7,5){$w_b$}
    \put(3.7,6.7){$w_c$}
    \put(1.7,5){$w_d$}
    \put(2.2,3.9){${\rm A}_j$}
    \put(5.3,3.9){${\rm B}_j$}
    \put(5.3,5.9){${\rm C}_j$}
    \put(2.2,5.9){${\rm D}_j$}
    \put(4.3,5){$j$}
  \end{picture}\hspace{2cm}\includegraphics[scale=0.6]{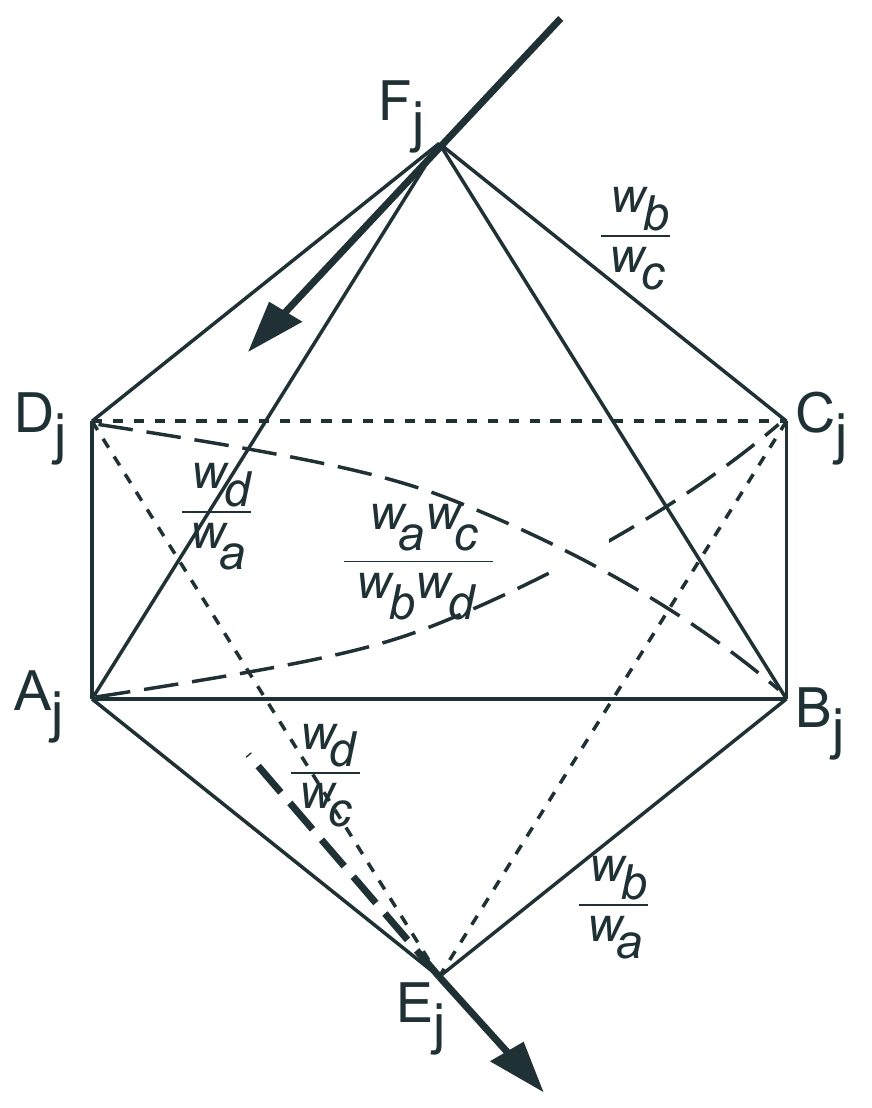}}\\
  \subfigure[Negative crossing $j$]
  {\begin{picture}(6,7)  
  \setlength{\unitlength}{0.8cm}\thicklines
        \put(4,5){\arc[5](1,1){360}}
    \put(2,7){\vector(1,-1){4}}
   \put(6,7){\line(-1,-1){1.8}}
   \put(3.8,4.8){\vector(-1,-1){1.8}}
    \put(3.7,3.2){$w_a$}
    \put(5.7,5){$w_b$}
    \put(3.7,6.7){$w_c$}
    \put(1.7,5){$w_d$}
    \put(2.2,3.9){${\rm D}_j$}
    \put(5.3,3.9){${\rm A}_j$}
    \put(5.3,5.9){${\rm B}_j$}
    \put(2.2,5.9){${\rm C}_j$}
    \put(4.3,5){$j$}
  \end{picture}\hspace{2cm}\includegraphics[scale=0.6]{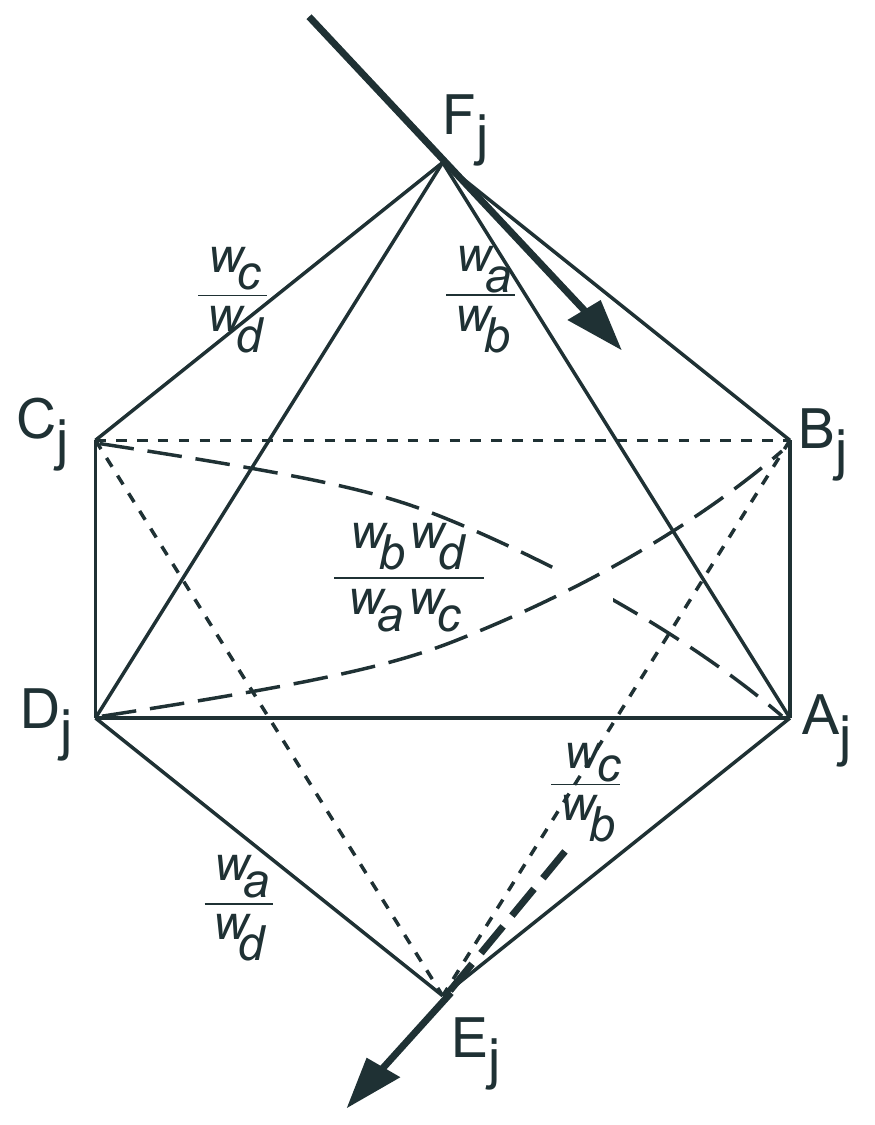}}
  \caption{Assignment of shape parameters}\label{fig7}
\end{figure}

To obtain the boundary parabolic representation 
$\rho:\pi_1(\mathbb{S}^3\backslash (L\cup\{\pm\infty\}))\longrightarrow{\rm PSL}(2,\mathbb{C})$,
we require two conditions on the ideal triangulation of $\mathbb{S}^3\backslash (L\cup\{\pm\infty\})$;
the product of shape parameters on any edge in the triangulation becomes one, and the holonomies induced
by meridian and longitude of the boundary torus act as {non-trivial} translations on the torus cusps.

Note that these conditions are all expressed as equations of shape parameters.
The former equations are called {\it (Thurston's) gluing equations}, the latter is called {\it completeness condition},
and the whole set of these equations are called {\it the hyperbolicity equations}.
Using Yoshida's construction in Section 4.5 of \cite{Tillmann13}, 
an {essential} solution ${\bold w}^{(0)}$ of the hyperbolicity equations determines a representation 
$$\rho_{\bold{w}^{(0)}}:\pi_1(\mathbb{S}^3\backslash (L\cup\{\pm\infty\}))
=\pi_1(\mathbb{S}^3\backslash L)\longrightarrow{\rm PSL}(2,\mathbb{C}).$$

\begin{proposition}[Proposition 1.1 of \cite{Cho13c}]\label{pro21} The set $\mathcal{I}$ defined in (\ref{defH}) is the set of the hyperbolicity equations of the five-term triangulation, where the shape parameters are assigned as in Figure \ref{fig7}. 
\end{proposition}

{The proof of this proposition is quite complicate and technical, so we refer \cite{Cho13c} and \cite{Cho13b}.
The following lemma was stated and used in \cite{Cho13c} without proof because it is almost trivial.
To avoid confusion, we add the proof here.

\begin{lemma}\label{ntri} The holonomy of a meridian induced by an essential solution 
$\bold{w}^{(0)}=(w_1^{(0)},\ldots,w_n^{(0)})$ of $\mathcal{I}$ is always non-trivial.
\end{lemma}

\begin{proof} At first, note that we assumed any component of the link diagram contains the gluing pattern
in Figure \ref{glue pattern}(a). (See Footnote 4.) For the local diagram with the meridian loop $m$ and the variables $w_a, w_b$
in Figure \ref{nontricusp}(a),
the corresponding cusp diagram of the five-term triangulation becomes Figure \ref{nontricusp}(b).

\begin{figure}[ht]
\centering
  \subfigure[Local diagram of the link]
  {\includegraphics[scale=0.6]{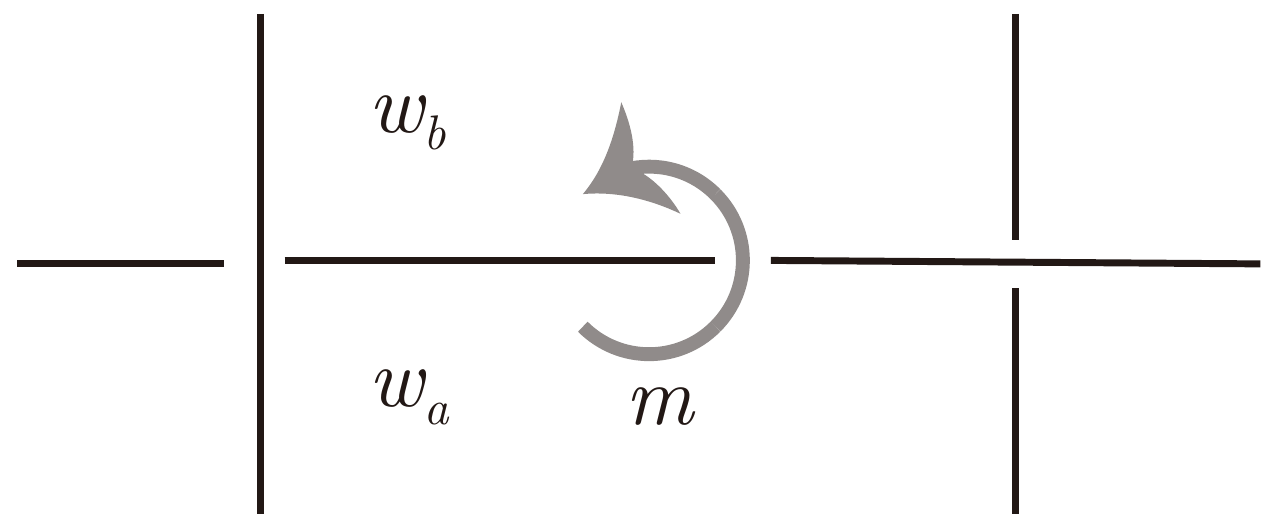}}\hspace{0.5cm}
  \subfigure[Local cusp diagram]
  {\includegraphics[scale=0.6]{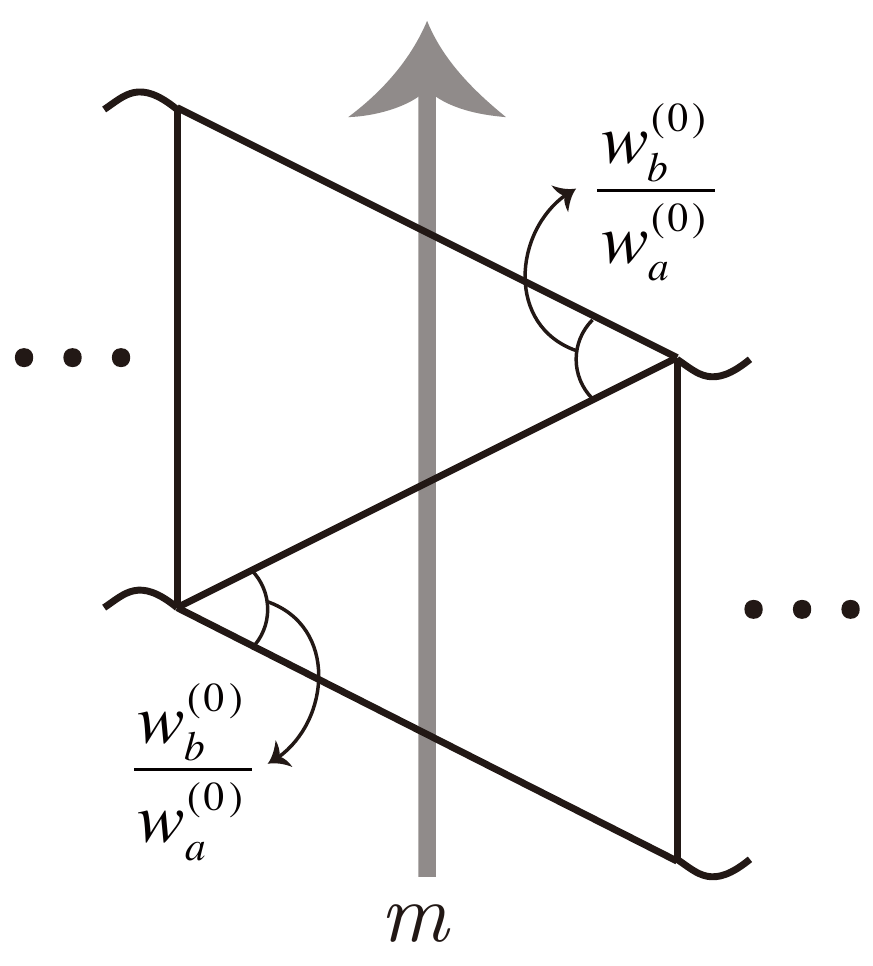}}
  \caption{Holonomy of the meridian $m$}\label{nontricusp}
\end{figure} 

Note that the same shape parameter $\frac{w_b^{(0)}}{w_a^{(0)}}$ is placed in two different positions in Figure \ref{nontricusp}(b).
The essentialness of the solution $\bold{w}^{(0)}$ guarantees $\frac{w_b^{(0)}}{w_a^{(0)}}\neq 0,1,\infty$,
so the holonomy cannot be trivial.

\end{proof}}

From Proposition \ref{pro21} {and Lemma \ref{ntri}}, a solution $\bold{w}^{(0)}=(w_1^{(0)},\ldots,w_n^{(0)})$ of $\mathcal{I}$ induces
a boundary-parabolic representation $\rho_{\bold{w^{(0)}}}$. We will make a special solution $\bold{w}^{(0)}$
from the given representation $\rho$ in the next section, which satisfies
$\rho_{\bold{w^{(0)}}}=\rho$ up to conjugation.

\section{Construction of the solution}\label{sec3}

\subsection{Reviews on shadow-coloring}

This section is a summary of definitions and properties we need. 
For complete descriptions, see Section 2 of \cite{Cho14a}.
(All definitions and Lemma \ref{lem} originally came from \cite{Kabaya14}.)

Let $\mathcal{P}$ be the set of parabolic elements of ${\rm PSL}(2,\mathbb{C})={\rm Isom^+}(\mathbb{H}^3)$.
We identify $\mathbb{C}^2\backslash\{0\}/\pm$ with $\mathcal{P}$ by
\begin{equation}\label{matrixcc}
\left(\begin{array}{cc}\alpha &\beta\end{array}\right)  
  \longleftrightarrow\left(\begin{array}{cc}1+\alpha\beta & \beta^2 \\ -\alpha^2& 1-\alpha\beta\end{array}\right),
\end{equation}
and define operation $*$ by
\begin{eqnarray*}
  \left(\begin{array}{cc}\alpha & \beta\end{array}\right)*  \left(\begin{array}{cc}\gamma & \delta\end{array}\right)
  :=\left(\begin{array}{cc}\alpha & \beta\end{array}\right)
  \left(\begin{array}{cc}1+\gamma\delta &\delta^2 \\   -\gamma^2& 1-\gamma\delta\end{array}\right)
  \in \mathcal{P},
\end{eqnarray*}
where this operation is actually induced by the conjugation as follows:
$$  \left(\begin{array}{cc}\alpha & \beta\end{array}\right)*  \left(\begin{array}{cc}\gamma & \delta\end{array}\right)\in\mathcal{P}
\longleftrightarrow \left(\begin{array}{cc}\gamma&\delta\end{array}\right)
  \left(\begin{array}{cc}\alpha & \beta\end{array}\right)
  \left(\begin{array}{cc}\gamma &\delta\end{array}\right)^{-1}\in{\rm PSL}(2,\mathbb{C}).$$
The inverse operation $*^{-1}$ is expressed by
$$\left(\begin{array}{cc}\alpha & \beta\end{array}\right)*^{-1}  \left(\begin{array}{cc}\gamma & \delta\end{array}\right)
  = \left(\begin{array}{cc}\alpha & \beta\end{array}\right)
  \left(\begin{array}{cc}1-\gamma\delta & -\gamma^2 \\ \delta^2 & 1+\gamma\delta\end{array}\right)\in\mathcal{P},$$
and $(\mathcal{P},*)$ becomes a {\it conjugation quandle}. 
Here, {\it quandle} implies, for any $a,b,c\in\mathcal{P}$, the map $*b:a\mapsto a*b$ is bijective and
$$a*a=a, ~(a*b)*c=(a*c)*(b*c)$$
hold. {\it Conjugation quandle} implies the operation $*$ is defined by the conjugation.

We define {\it the Hopf map} $h:\mathcal{P}\rightarrow\mathbb{CP}^1=\mathbb{C}\cup\{\infty\}$ by
$$\left(\begin{array}{cc}\alpha &\beta\end{array}\right)\mapsto \frac{\alpha}{\beta}.$$
Note that the image 
is the fixed point of the M\"{o}bius transformation $f(z)=\frac{(1+\alpha\beta)z-\alpha^2}{\beta^2 z+(1-\alpha\beta)}$.

For an oriented link diagram $D$ of $L$ and a given boundary-parabolic representation $\rho$, we assign {\it arc-colors}
$a_1,\ldots,a_r\in\mathcal{P}$ to arcs of $D$ so that each $a_k$ is the image of the meridian around the arc
under the representation $\rho$. Note that, in Figure \ref{fig02}, we have 
\begin{equation}\label{ope}
     a_m=a_l*a_k.
\end{equation}

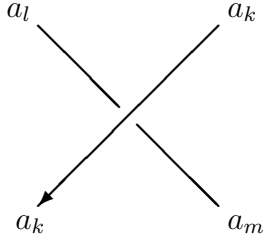
\begin{figure}[h]
\centering  \setlength{\unitlength}{0.6cm}\thicklines
\begin{picture}(4.5,5.5)(1.5,0.5)  
    \put(6,5){\vector(-1,-1){4}}
    \put(2,5){\line(1,-1){1.8}}
    \put(4.2,2.8){\line(1,-1){1.8}}
    \put(6.2,5.2){$a_k$}
    \put(1.5,0.5){$a_k$}
    \put(1.3,5.2){$a_l$}
    \put(6.2,0.5){$a_m$}
  \end{picture}
  \caption{Arc-coloring}\label{fig02}
\end{figure}

We also assign {\it region-colors} $s_1,\ldots,s_n\in\mathcal{P}$ to regions of $D$ satisfying the rule in Figure \ref{fig03}.
Note that, if an arc-coloring is given, then a choice of one region-color determines all the other region-colors.

\begin{figure}[h]
\centering  \setlength{\unitlength}{0.6cm}\thicklines
\begin{picture}(6,5)  
    \put(6,4){\vector(-1,-1){4}}
    \put(1.5,2.2){$s$}
    \put(5,1.5){$s*a_k$}
    \put(6.2,4.2){$a_k$}
  \end{picture}
  \caption{Region-coloring}\label{fig03}
\end{figure}
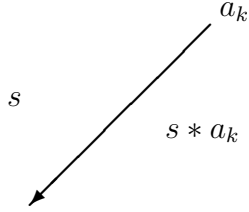

\begin{lemma}[Lemma 2.4 of \cite{Cho14a}]\label{lem} 
Consider the arc-coloring induced by the boundary-parabolic representation $\rho:\pi_1(L)\rightarrow {\rm PSL}(2,\mathbb{C})$.
Then, for any triple $(a_k,s,s*a_k)$ of an arc-color $a_k$ and its surrounding region-colors $s, s*a_k$ as in Figure \ref{fig03},
there exists a region-coloring satisfying
\begin{equation*}
  h(a_k)\neq h(s)\neq h(s*a_k)\neq h(a_k).
\end{equation*}
\end{lemma}

\begin{proof}
For the given arc-colors $a_1,\ldots,a_r$, we choose region-colors $s_1,\ldots,s_n$ so that
\begin{equation}\label{exi}
  \{h(a_1),\ldots,h(a_r)\}\cap\{h(s_1),\ldots,h(s_n)\}=\emptyset.
\end{equation}
This is always possible because, {
each $h(s_k)$ is written as $h(s_k)=M_k(h(s_1))$ by a M\"{o}bius transformation $M_k$, which only depends on
the arc-colors $a_1,\dots,a_r$. If we choose $h(s_1)\in\mathbb{CP}^1$ away from the finite set
$$\bigcup_{1\leq k\leq n}\left\{M_k^{-1}(h(a_1)),\ldots,M_k^{-1}(h(a_r))\right\},$$
we have $h(s_k)\notin \{h(a_1),\ldots,h(a_r)\}$ for all $k$.}
 
Now consider Figure \ref{fig03} and assume
$h(s*a_k)=h(s)$. Then we obtain
\begin{equation}\label{eqnh}
  h(s*a_k)=\widehat{a_k}(h(s))=h(s),
\end{equation}
where $\widehat{a_k}:\mathbb{CP}^1\rightarrow\mathbb{CP}^1$ is the M\"{o}bius transformation
\begin{equation}\label{mob}
\widehat{a_k}(z)=\frac{(1+\alpha_k\beta_k)z-\alpha_k^2}{\beta_k^2 z+(1-\alpha_k\beta_k)}
\end{equation}
of $a_k=\left(\begin{array}{cc}\alpha_k &\beta_k\end{array}\right)$. Then (\ref{eqnh}) implies
$h(s)$ is the fixed point of $\widehat{a_k}$, which means $h(a_k)=h(s)$ and this contradicts (\ref{exi}).

\end{proof}

We remark that Lemma \ref{lem} holds for any choice of $h(s_1)\in\mathbb{CP}^1$ with only finitely many exceptions. 
Therefore, if we want to find a region-coloring explicitly, 
we first choose $h(s_1)\notin\{h(a_1),\ldots,h(a_r)\}$ and then decide $h(s_2), \ldots, h(s_n)$ using 
\begin{equation}\label{app1}
h(s_1*a)=\widehat{a}(h(s_1)),~h(s_1*^{-1} a)=\widehat{a}^{-1}(h(s_1)).
\end{equation}
If this choice does not satisfy Lemma \ref{lem}, then we change $h(s_1)$ and do the same process again.
This process is very easy and it ends in finite steps. If proper $h(s_1)$ is chosen, 
then we can easily extend the value $h(s_1)$ to a region-color $s_1$ and find the proper region-coloring $\{s_1,\ldots,s_n\}$.
This observation implies the following corollary.

\begin{corollary}\label{cor32}
 Consider a sequence of diagrams $D_1,\ldots,D_m$, where each $D_{k+1}$ ($k=1,\ldots,m-1$) is obtained from $D_k$ by applying
one of the Reidemeister moves in Figure \ref{pic02} once.
Also assume arc-colorings of $D_1,\ldots,D_m$ are given by certain boundary-parabolic representation 
$\rho:\pi_1(L)\rightarrow {\rm PSL}(2,\mathbb{C})$.
(Note that a region-coloring of $D_1$ determines the region-colorings of $D_2,\ldots,D_m$ uniquely.)
Then there exists a region-coloring of $D_1$ satisfying Lemma \ref{lem} for all region-colorings of $D_1,\ldots,D_m$.
\end{corollary}

\begin{proof}
Let $s_1$ be the region-color of the unbounded region of $D_1,\ldots,D_m$.
For each $D_k$, the number of values of $h(s_1)\in\mathbb{CP}^1$ that does not satisfy Lemma \ref{lem} is finite.
Therefore, we can choose $h(s_1)$ so that Lemma \ref{lem} holds for all $D_1,\ldots,D_m$.
By extending $h(s_1)$ to a region-color $s_1$, we can determine the region-colorings of $D_1,\ldots,D_m$
satisfying Lemma \ref{lem} uniquely.
 
\end{proof}

The arc-coloring induced by $\rho$ together with the region-coloring satisfying Lemma \ref{lem}
is called {\it the shadow-coloring induced by} $\rho$. 
We choose an element $p\in\mathcal{P}$ so that 
\begin{equation}\label{p}
h(p)\notin\{h(a_1),\ldots,h(a_r), h(s_1),\ldots,h(s_n)\}.
\end{equation}
The geometric shape of the five-term triangulation in Section \ref{sec2} will be determined by the shadow-coloring induced by $\rho$ and $p$ 
in the next section.

From now on, we fix the representatives of shadow-colors in $\mathbb{C}^2\backslash\{0\}$.
As mentioned in \cite{Cho14a}, the representatives of some arc-colors may satisfy (\ref{ope}) up to sign, 
in other words, $a_m=\pm (a_l*a_k)$. 
However, the representatives of the region-colors are uniquely determined due to the fact $s*(\pm a)=s*a$
for any region-color $s$ and any arc-color $a$.

For $a=\left(\begin{array}{cc}\alpha_1 &\alpha_2\end{array}\right)$ and $
b=\left(\begin{array}{cc}\beta_1 &\beta_2\end{array}\right)$ in $\mathbb{C}^2\backslash\{0\}$,
we define {\it determinant} $\det(a,b)$ by
\begin{equation*}
  \det(a,b):=\det\left(\begin{array}{cc}\alpha_1 & \alpha_2 \\\beta_1 & \beta_2\end{array}\right)=\alpha_1 \beta_2-\alpha_2 \beta_1.
\end{equation*}
Then the determinant satisfies
$\det(a*c,b*c)=\det(a,b)$
for any $a,b,c\in\mathbb{C}^2\backslash\{0\}$. Furthermore, for $v_0,\ldots,v_3\in\mathbb{C}^2\backslash \{0\}$,
the cross-ratio can be expressed using the determinant by
$$[h(v_0),h(v_1),h(v_2),h(v_3)]=\frac{\det(v_0,v_3)\det(v_1,v_2)}{\det(v_0,v_2)\det(v_1,v_3)}.$$
(For the proofs, see Section 2 of \cite{Cho14a}.)

\subsection{Geometric shape of the five-term triangulation}\label{sec22}

Note that the five-term triangulation was already defined in Section \ref{sec2}.
Consider the crossings in Figure \ref{fig04} with the shadow-colorings induced by $\rho$, and let $w_a,\ldots,w_d$
be the variables assigned to regions of $D$. 

\begin{figure}[h]
\centering
\subfigure[Positive crossing]{\begin{picture}(6,5.5)  
  \setlength{\unitlength}{0.8cm}\thicklines
    \put(6,5){\vector(-1,-1){4}}
    \put(2,5){\line(1,-1){1.8}}
    \put(4.2,2.8){\vector(1,-1){1.8}}
    \put(6.2,5.2){$a_k$}
    \put(1.3,0.8){$a_k$}
    \put(1.3,5.2){$a_l$}
    \put(6.2,0.8){$a_l*a_k$}
    \put(3.5,4.7){$s*a_l$}
    \put(3.9,4.2){$w_c$}
    \put(2,3){$s$}
    \put(1.9,2.5){$w_d$}
    \put(5.5,3){\small $(s*a_l)*a_k$}
    \put(6.5,2.5){$w_b$}
    \put(3.5,1){$s*a_k$}
    \put(3.9,0.5){$w_a$}
  \end{picture}}
\subfigure[Negative crossing]{\begin{picture}(6,5.5)  
  \setlength{\unitlength}{0.8cm}\thicklines
    \put(6,5){\vector(-1,-1){4}}
    \put(3.8,3.2){\vector(-1,1){1.8}}
    \put(4.2,2.8){\line(1,-1){1.8}}
    \put(6.2,5.2){$a_k$}
    \put(1.3,0.8){$a_k$}
    \put(1.3,5.2){$a_l$}
    \put(6.2,0.8){$a_l*a_k$}
    \put(4,4.7){$s$}
    \put(3.9,4.2){$w_d$}
    \put(1.2,3){$s*a_l$}
    \put(1.5,2.5){$w_a$}
    \put(5.5,3){$s*a_k$}
    \put(5.8,2.5){$w_c$}
    \put(3,1){\small $(s*a_l)*a_k$}
    \put(3.9,0.5){$w_b$}
  \end{picture}}
 \caption{Crossings with shadow-colorings and the region variables}\label{fig04}
\end{figure}
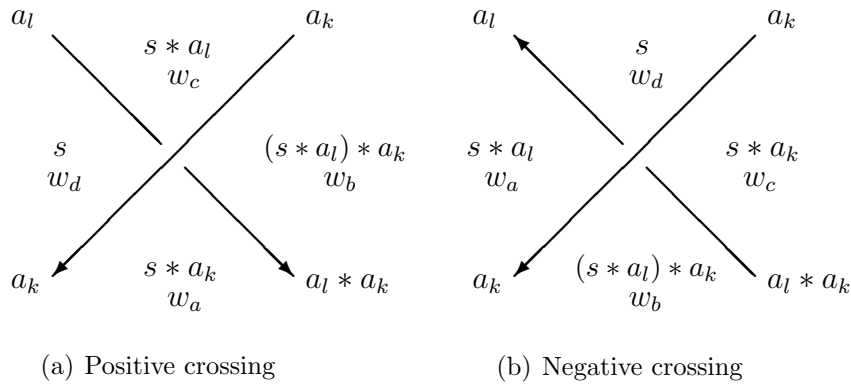

We place tetrahedra at each crossings of $D$ and assign coordinates of them as in Figure \ref{fig06}
so as to make them hyperbolic ideal tetrahedra in the upper-space model of the hyperbolic 3-space $\mathbb{H}^3$.

\begin{figure}[h]
\centering
\subfigure[Positive crossing]{\includegraphics[scale=0.8]{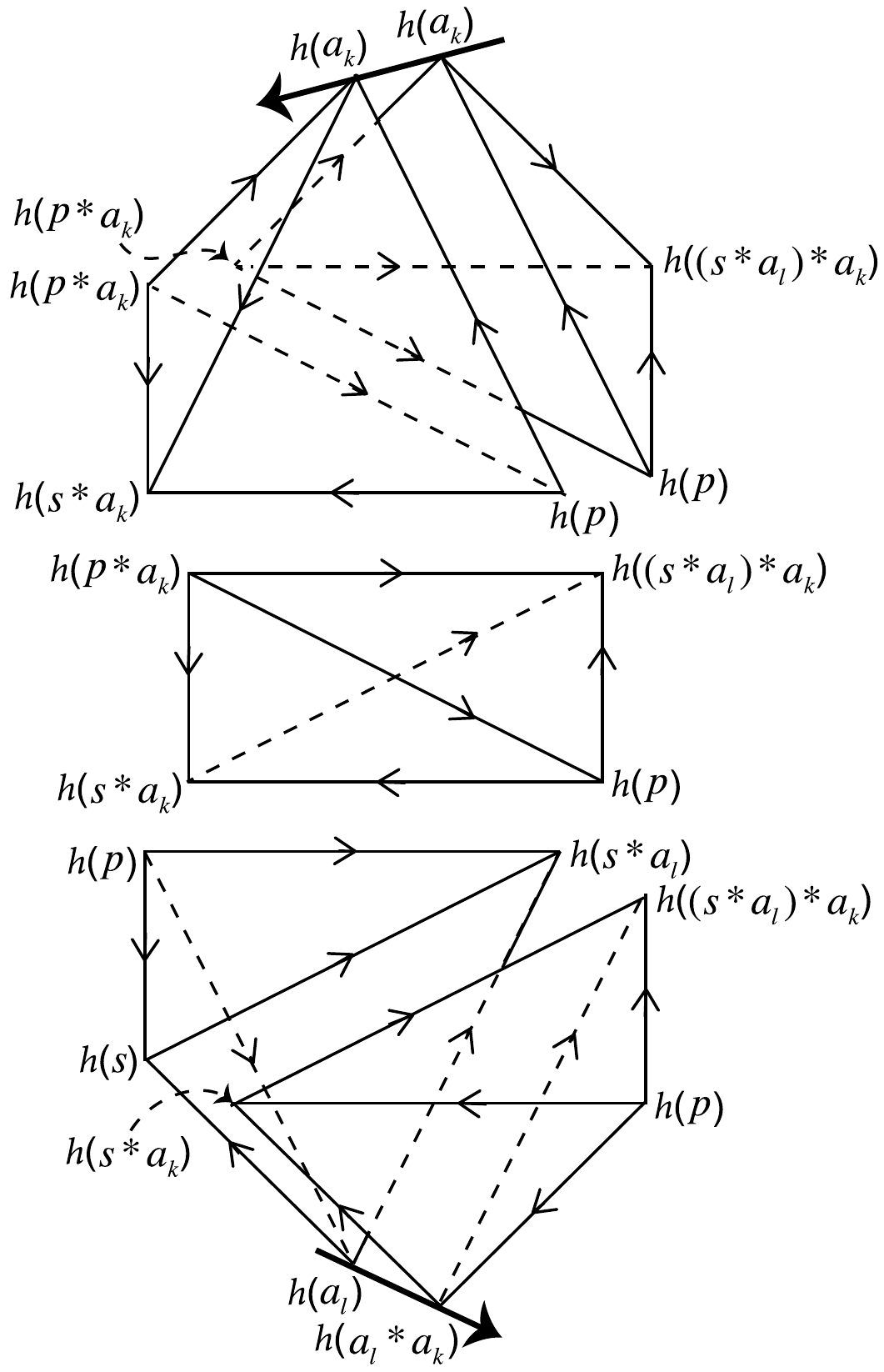}}
\subfigure[Negative crossing]{\includegraphics[scale=0.8]{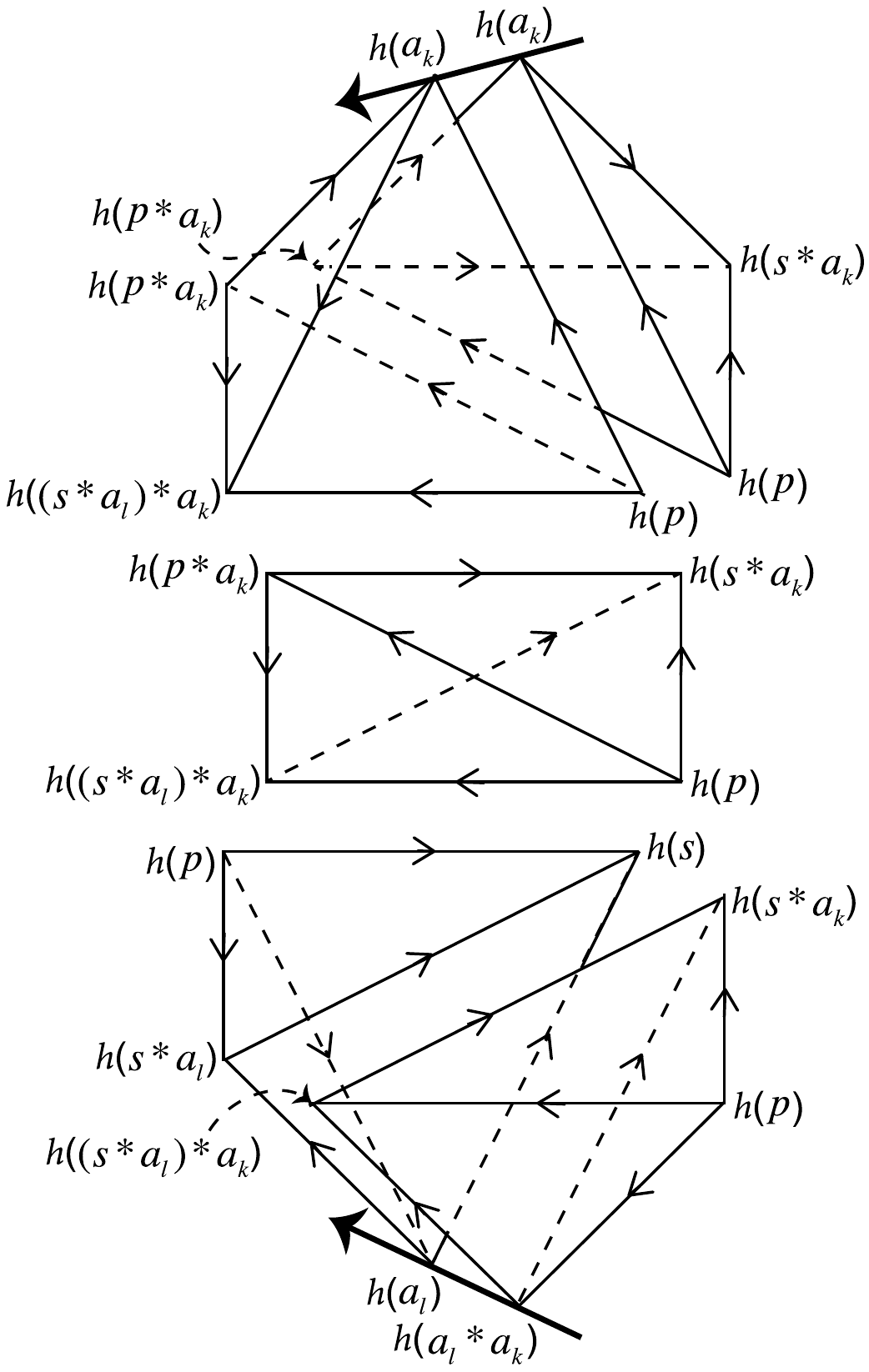}}
 \caption{Five-term triangulation at the crossing in Figure \ref{fig04}}\label{fig06}
\end{figure}

\begin{proposition}[Proposition 2.3 of \cite{Cho14c}]\label{pro} All the tetrahedra in Figure \ref{fig06} are non-degenerate.
\end{proposition}


According to Section 2 of \cite{Kabaya12} and the proof of Theorem 5 of \cite{Kabaya14}, the five-term triangulation defined by Figure \ref{fig06}
induces the given representation\footnote{The triangulation of \cite{Kabaya12} is different from ours.
However, the fundamental domain obtained by the five-term triangulation coincides with that of \cite{Kabaya12}, 
so it induces the same representation. (Our triangulation is obtained by choosing different subdivision of the same fundamental domain. 
See Section 2.2 of \cite{Cho14c} for details.)} 
$\rho:\pi_1(L)\rightarrow{\rm PSL}(2,\mathbb{C})$ and the shape parameters of this triangulation satisfy
the gluing equations of all edges. (The face-pairing maps are the isomorphisms induced by 
the M\"{o}bius transformations of ${a_1},\ldots,a_r\in{\rm PSL}(2,\mathbb{C})$. {Note that this construction is based on
the construction method developed at \cite{Neumann99} and \cite{Zickert09}.})
Furthermore, the representation $\rho$ is boundary-parabolic, 
which implies the shape-parameters satisfy the hyperbolicity equations of the five-term triangulation.


\subsection{Formula of the solution ${\bold w}^{(0)}$}\label{sec33}
Consider the crossings in Figure \ref{fig04} and the tetrahedra in Figure \ref{fig06}. For the positive crossing,
we assign shape parameters to the edges as follows:
\begin{itemize} 
\item $\displaystyle\frac{w_d}{w_a}$ to $(h(a_k), h(s*a_k))$ of $(h(p*a_k),h(p),h(a_k), h(s*a_k))$, 
\item $\displaystyle\frac{w_b}{w_c}$ to $(h(a_k), h((s*a_l)*a_k))$ of $-(h(p*a_k),h(p),h(a_k), h((s*a_l)*a_k))$, 
\item $\displaystyle\frac{w_b}{w_a}$ to $(h(p), h(a_l*a_k))$ of $(h(p), h(a_l*a_k),h(s*a_k), h((s*a_l)*a_k))$ and 
\item $\displaystyle\frac{w_d}{w_c}$ to $(h(p), h(a_l))$ of $-(h(p), h(a_l),h(s), h(s*a_l))$, respectively.
\end{itemize}
On the other hand, for the negative crossing,
we assign shape parameters to the edges as follows:
\begin{itemize} 
\item $\displaystyle\frac{w_a}{w_b}$ to $(h(a_k), h((s*a_l)*a_k))$ of $-(h(p),h(p*a_k),h(a_k), h((s*a_l)*a_k))$, 
\item $\displaystyle\frac{w_c}{w_d}$ to $(h(a_k), h(s*a_k))$ of $(h(p),h(p*a_k),h(a_k), h(s*a_k))$, 
\item $\displaystyle\frac{w_c}{w_b}$ to $(h(p), h(a_l*a_k))$ of $(h(p), h(a_l*a_k), h((s*a_l)*a_k),h(s*a_k))$ and 
\item $\displaystyle\frac{w_a}{w_d}$ to $(h(p), h(a_l))$ of $-(h(p), h(a_l),h(s*a_l), h(s))$, respectively.
\end{itemize}
Note that these assignments coincide with the one defined in Figure \ref{fig7}.

\begin{proposition}[Theorem 1.1 of \cite{Cho14c}]\label{pro2}
For a region of $D$ with region-color $s_k$ and region-variable $w_k$, we define
\begin{equation}\label{main}
w_k^{(0)}:=\det(p,s_k).
\end{equation}
Then, 
$w_k^{(0)}\not= 0$ 
and ${\bold w^{(0)}}=(w_1^{(0)},\ldots,w_n^{(0)})$ is an {essential} solution of the hyperbolicity equations in $\mathcal{I}$ of (\ref{defH}). 
Furthermore, the solution satisfies $\rho_{\bold w^{(0)}}=\rho$ up to conjugation and
\begin{equation}\label{volume}
W_0(\bold w^{(0)})\equiv i(\vol(\rho)+i\,\cs(\rho))~~({\rm mod}~\pi^2).
\end{equation}
\end{proposition}

\begin{proof}

The first property $w_k^{(0)}\not= 0$ is trivial from the definition of $p$ in (\ref{p}).
From the discussion below Proposition \ref{pro},
the shape parameters of the five-term triangulation defined by Figure \ref{fig06} satisfy the hyperbolicity equations
and the fundamental domain induces the boundary-parabolic representation $\rho$.

On the other hand, direct calculation shows the values $w_k^{(0)}$ defined in (\ref{main}) determines 
the same shape parameter of the five-term triangulation defined by Figure \ref{fig06}. 
Specifically, for the first two cases of the positive crossing, the shape parameters
assigned to edges $(h(a_k), h(s*a_k))$ and $(h(a_k), h((s*a_l)*a_k))$ are the cross-ratios
\begin{eqnarray*}
[h(p*a_k),h(p),h(a_k), h(s*a_k)]&=&\frac{\det(p*a_k,s*a_k)\det(p,a_k)}{\det(p*a_k,a_k)\det(p,s*a_k)}\\
&=&\frac{\det(p,s)\det(p,a_k)}{\det(p,a_k)\det(p,s*a_k)}=\frac{w_d^{(0)}}{w_a^{(0)}},\\
{[h}(p*a_k),h(p),h(a_k), h((s*a_l)*a_k)]^{-1}
&=&\frac{\det(p*a_k,a_k)\det(p,(s*a_l)*a_k)}{\det(p*a_k,(s*a_l)*a_k)\det(p,a_k)}\\
&=&\frac{\det(p,a_k)\det(p,(s*a_l)*a_k)}{\det(p,s*a_l)\det(p,a_k)}=\frac{w_b^{(0)}}{w_c^{(0)}},
\end{eqnarray*}
respectively, and all the other cases can be verified by the same way. 
From Proposition \ref{pro21}, we conclude that ${\bold w^{(0)}}=(w_1^{(0)},\ldots,w_n^{(0)})$ is an {essential} solution of $\mathcal{I}$.

Finally, the identity (\ref{volume}) was already proved in Theorem 1.2 of \cite{Cho13c}.

\end{proof} 

{In Appendix \ref{app}, we will show that any {essential} solution of $\mathcal{I}$ can be constructed by certain shadow-coloring.}
 
\section{Reidemeister transformations on the solution}\label{sec4}

In this section, we show how the solution ${\bold w^{(0)}}$ of $\mathcal{I}$ defined in Proposition \ref{pro2}
changes under the Reidemeister moves. We assume all the region-colorings in this and later sections satisfy Corollary \ref{cor32}
{so that the original and the transformed solutions are both essential}.

At first, we introduce very simple, but useful lemma.
{Recall that, according to Proposition \ref{pro21}, the set $\mathcal{I}$ defined in (\ref{defH}) is the set of the hyperbolicity equations.} The following lemma shows the hyperbolicity equations do not change
under the change of the orientation.

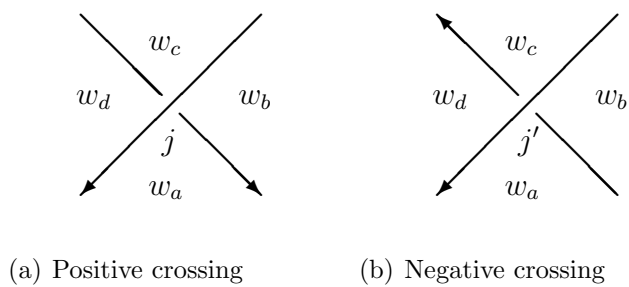
\begin{figure}[h]
\centering\setlength{\unitlength}{0.6cm}  \thicklines
\subfigure[Positive crossing]{\begin{picture}(6,5.5)
    \put(6,5){\vector(-1,-1){4}}
    \put(2,5){\line(1,-1){1.8}}
    \put(4.2,2.8){\vector(1,-1){1.8}}
    \put(3.8,2){$j$}
    \put(3.5,4.2){$w_c$}
    \put(1.9,3){$w_d$}
    \put(3.5,1){$w_a$}
    \put(5.5,3){$w_b$}
  \end{picture}}\hspace{1cm}
\subfigure[Negative crossing]{\begin{picture}(6,5.5)  
  \setlength{\unitlength}{0.6cm}\thicklines
    \put(6,5){\vector(-1,-1){4}}
    \put(3.8,3.2){\vector(-1,1){1.8}}
    \put(4.2,2.8){\line(1,-1){1.8}}
    \put(3.8,2){$j'$}
    \put(3.5,4.2){$w_c$}
    \put(1.9,3){$w_d$}
    \put(3.5,1){$w_a$}
    \put(5.5,3){$w_b$}
  \end{picture}}
 \caption{Change of orientation}\label{fig10}
\end{figure}

\begin{lemma}\label{lem40}
For the potential function $W^j(w_a,w_b,w_c,w_d)$ and $W^{j'}(w_a,w_b,w_c,w_d)$ of Figure \ref{fig10}(a) and (b), respectively,
we have
$$\exp(w_k\frac{\partial W^j}{\partial w_k})=\exp(w_k\frac{\partial W^{j'}}{\partial w_k}),$$
for $k=a,b,c,d$.
\end{lemma}

\begin{proof}
It is easily verified by direct calculation. For example,
$$\exp(w_a\frac{\partial W^j}{\partial w_a})=\frac{(w_b-w_a)(w_d-w_a)}{w_bw_d-w_cw_a}
=\exp(w_a\frac{\partial W^{j'}}{\partial w_a}).$$

\end{proof}

\subsection{Reidemeister 1st move}

Consider the Reidemeister 1st moves in Figure \ref{LocalMove1}. 
Let $\alpha\in\mathcal{P}$ be the arc-color, $s,s*\alpha,(s*\alpha)*\alpha\in\mathcal{P}$ be 
the region-colors and $w_a,w_b,w_c$ be the variables of the potential function.
Then, by (\ref{main}),
\begin{equation}\label{eq41}
w_a^{(0)}=\det(p,s),~w_b^{(0)}=\det(p,s*\alpha), ~w_c^{(0)}=\det(p,(s*\alpha)*\alpha).
\end{equation}

\begin{figure}[h]\centering
\includegraphics[scale=1.2]{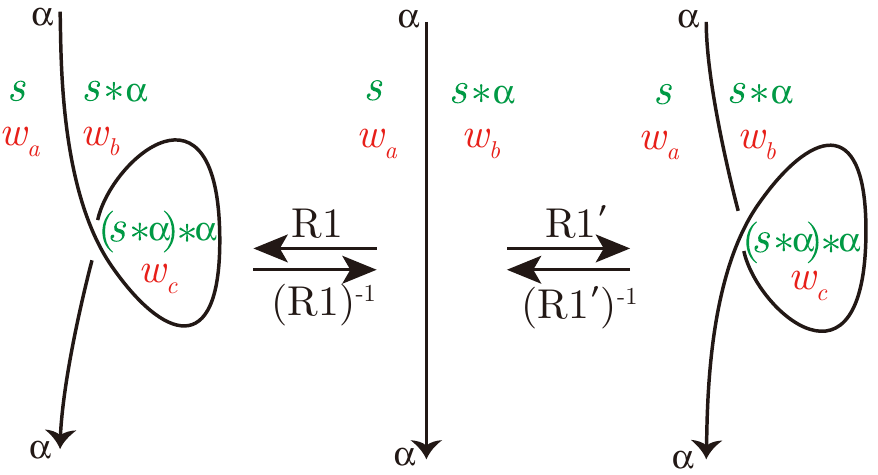}
 \caption{First Reidemeister moves}\label{LocalMove1}
\end{figure}

\begin{lemma}\label{LemR1}
The values $w_a^{(0)},w_b^{(0)},w_c^{(0)}$ defined in (\ref{eq41}) satisfy
$$w_c^{(0)}=2w_b^{(0)}-w_a^{(0)}.$$
\end{lemma}

\begin{proof}
Using the identification (\ref{matrixcc}), let
\begin{equation*}
  \alpha=\left(\begin{array}{cc}\alpha_1 & \alpha_2\end{array}\right)  
  \longleftrightarrow A=\left(\begin{array}{cc}1+\alpha_1\alpha_2 &\alpha_2^2  \\-\alpha_1^2 & 1-\alpha_1\alpha_2\end{array}\right).
\end{equation*}
Then $s*\alpha=sA\in\mathcal{P}$ and $(s*\alpha)*\alpha=sA^2\in\mathcal{P}$ holds by the definition of the operation $*$. 
Furthermore, by the Cayley-Hamilton theorem, the matrix $A$ satisfies
$$A^2-2A+I=0,$$
where $I$ is the $2\times 2$ identity matrix. Using these, we obtain
$$w_c^{(0)}=\det(p,(s*\alpha)*\alpha)=\det(p,sA^2)=2\det(p,sA )-\det(p,s I )=2w_b^{(0)}-w_a^{(0)}.$$
\end{proof}

\subsection{Reidemeister 2nd moves}

Consider the Reidemeister 2nd moves in Figure \ref{LocalMove2}. 
Let $\alpha$, $\beta$, $\alpha*\beta\in\mathcal{P}$ be the arc-colors, 
$s$, $s*\alpha$, $s*\beta$, $(s*\alpha)*\beta\in\mathcal{P}$ be 
the region-colors and $w_a,\ldots,w_e$ be the variables of the potential function.
Then, by (\ref{main}),
\begin{equation*}
w_a^{(0)}=\det(p,s*\beta),~w_b^{(0)}=w_d^{(0)}=\det(p,(s*\alpha)*\beta), ~w_c^{(0)}=\det(p,s*\alpha),~w_e^{(0)}=\det(p,s)
\end{equation*}
for the case of R2 move in Figure \ref{LocalMove2}(a), and
\begin{equation*}
w_a^{(0)}=\det(p,s*\alpha),~w_b^{(0)}=w_d^{(0)}=\det(p,s), ~w_c^{(0)}=\det(p,s*\beta),~w_e^{(0)}=\det(p,(s*\alpha)*\beta)
\end{equation*}
for the case of R2$'$ move in Figure \ref{LocalMove2}(b).

\begin{figure}[h]
\centering
\subfigure[R2 move]{\includegraphics[scale=1.2]{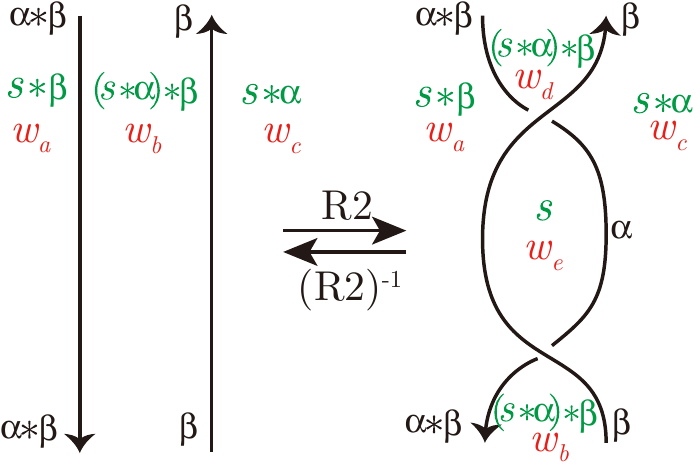}}\hspace{0.5cm}
\subfigure[R2$'$ move]{\includegraphics[scale=1.2]{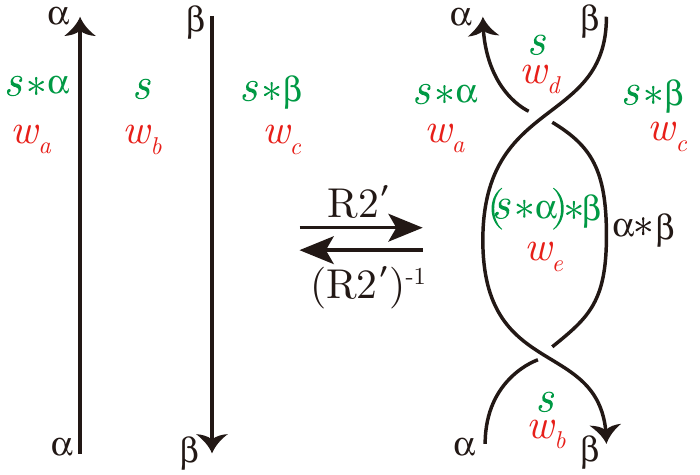}}
 \caption{Second Reidemeister moves}\label{LocalMove2}
\end{figure}

\begin{lemma}\label{LemR2} Let $T(W)(...,w_a,w_b,w_c,w_d,w_e,...)$ be the potential function (\ref{R2a}) (or (\ref{R2c})) of the diagram in Figure \ref{LocalMove2}(a) (or (b)) after applying $R2$ (or $R2'$) move. Then $w_d^{(0)}=w_b^{(0)}$ and $w_e^{(0)}$ is uniquely determined from the values of the parameters around the region of $w_b$ by the equation 
\begin{equation*}
\exp(w_b\frac{\partial T(W)}{\partial w_b})=1.
\end{equation*}
\end{lemma}

\begin{proof}
At first, $$\exp(w_e\frac{\partial T(W)}{\partial w_e})=\frac{w_aw_c-w_bw_e}{w_aw_c-w_dw_e}=1$$
induces $w_d^{(0)}=w_b^{(0)}$.

Consider the case of Figure \ref{LocalMove2}(a). The variable $w_e$ of the potential function $T(W)$ appears 
only inside the function $V_{R2}(w_a, w_b, w_c, w_d, w_e)$ defined in Section \ref{sec12}, and direct calculation shows
$$\exp(w_b\frac{\partial V_{R2}}{\partial w_b})
=\frac{w_a w_c-w_b w_e}{(w_c-w_b)(w_a-w_b)}.$$
Therefore, the equation $\exp(w_b\frac{\partial T(W)}{\partial w_b})=1$ is linear with respect to $w_e$
and it determines $w_e^{(0)}$ uniquely.

The case of Figure \ref{LocalMove2}(b) is trivial from Lemma \ref{lem40} and the above.

\end{proof}

Remark that the equation $\exp(w_d\frac{\partial T(W)}{\partial w_d})=1$ also determines the same value $w_e^{(0)}$. 

\subsection{Reidemeister 3rd move}

Consider the Reidemeister 3rd move in Figure \ref{LocalMove3}. 
Let $\alpha$, $\beta$, $\gamma$, $\beta*\gamma$, $(\alpha*\beta)*\gamma=(\alpha*\gamma)*(\beta*\gamma)\in\mathcal{P}$ be the arc-colors, 
$s$, $s*\alpha$, $s*\beta$, $s*\gamma$, $(s*\alpha)*\beta$, $(s*\alpha)*\gamma$, $(s*\beta)*\gamma$, $((s*\alpha)*\beta)*\gamma\in\mathcal{P}$ be 
the region-colors and $w_a,\ldots,w_g,w_h$ be the variables of the potential function.
Then, by (\ref{main}),
\begin{align}
w_a^{(0)}=\det(p,s),~w_b^{(0)}=\det(p,s*\gamma),~w_c^{(0)}=\det(p,(s*\beta)*\gamma),\label{eq43}\\
w_d^{(0)}=\det(p,((s*\alpha)*\beta)*\gamma),~w_e^{(0)}=\det(p,(s*\alpha)*\beta),\nonumber\\
w_f^{(0)}=\det(p,s*\alpha),~w_g^{(0)}=\det(p,s*\beta),~w_h^{(0)}=\det(p,(s*\alpha)*\gamma).\nonumber
\end{align}

\begin{figure}[h]\centering
\includegraphics[scale=1.3]{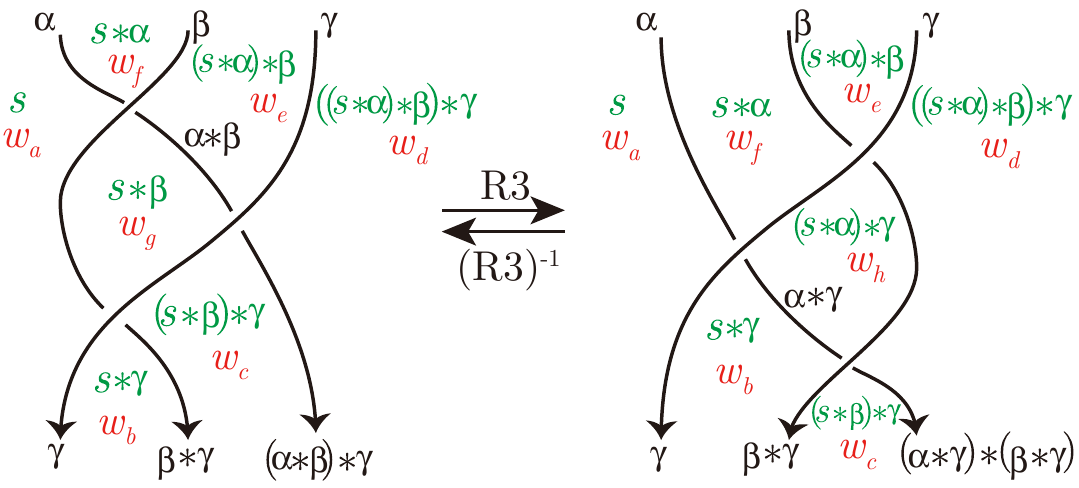}
 \caption{Third Reidemeister move}\label{LocalMove3}
\end{figure}

\begin{lemma}\label{LemR3}
The values $w_a^{(0)},\ldots,w_h^{(0)}$ defined in (\ref{eq43}) satisfy
$$w_d^{(0)}w_g^{(0)}-w_c^{(0)}w_e^{(0)}=w_a^{(0)}w_h^{(0)}-w_b^{(0)}w_f^{(0)}.$$
\end{lemma}

\begin{proof}
Using the identification (\ref{matrixcc}), let 
$$\alpha\leftrightarrow A,~ \beta\leftrightarrow B,~\gamma\leftrightarrow C,$$
{and
$$A=\left(\begin{array}{cc}1+a_1 a_2 &a_2^2  \\-a_1^2 & 1-a_1a_2\end{array}\right),
B=\left(\begin{array}{cc}1+b_1 b_2 &b_2^2   \\-b_1^2& 1-b_1b_2\end{array}\right),
C=\left(\begin{array}{cc}1+c_1 c_2 &  c_2^2\\ -c_1^2& 1-c_1c_2\end{array}\right).$$ 
Also, put
$$p=\left(\begin{array}{cc}p_1 &p_2\end{array}\right)\text{ and } s=\left(\begin{array}{cc}s_1 &s_2\end{array}\right).$$
Then direct calculation shows the following identity:
\begin{align*}
&\det(p,sABC)\det(p,sB)-\det(p,sBC)\det(p,sAB)\\
&=-(c_2p_1-c_1p_2)^2(a_2s_1-a_1s_2)^2\\
&=\det(p,s)\det(p,sAC)-\det(p,sC)\det(p,sA).
\end{align*}
(
Although this identity looks very elementary, the authors cannot find any other proof except the direct calculation.)
Applying (\ref{eq43}) to this identity proves the lemma.}

\end{proof}

\section{Orientation change and the mirror image}\label{sec5}

The proofs of the relations of solutions in Lemma \ref{LemR1}-\ref{LemR3} needed orientation of the link diagram.
However, we can show that the same relations still hold for any choice of orientations and for the mirror images.
(Exact statements will appear below.) These results are very useful when we consider the actual examples
because they reduce the number of the Reidemeister moves.


\begin{lemma}[Uniqueness of the solution]\label{uniquelemma}
Let $\bold{w}=(w_1^{(0)},\ldots,w_n^{(0)})$ be the solution of the hyperbolicity equations obtained by the shadow-coloring induced by $\rho$. (See Proposition \ref{pro2} for the construction.)
After applying one of the Reidemeister moves R1, R1$'$, R3 and (R3)$^{-1}$ to the link diagram once, 
assume the new variable $w_{n+1}$ appeared.
Then the value $w_{n+1}^{(0)}$ satisfying $(w_1^{(0)},\ldots,w_{n+1}^{(0)})$ to be a solution of the hyperbolicity equations
is uniquely determined by the values $w_1^{(0)},\ldots,w_n^{(0)}$.
Likewise, if new variables $w_{n+1}$ and $w_{n+2}$ appeared after applying R2 or R2$\,'$ move once,
then the values $w_{n+1}^{(0)}$ and $w_{n+2}^{(0)}$ satisfying the hyperbolicity equations
are uniquely determined by the values $w_1^{(0)},\ldots,w_n^{(0)}$.
\end{lemma}

\begin{proof} Note that the main idea was already appeared in the proof of Lemma \ref{LemR2}.

Let $T(W)$ be the potential function of the link diagram obtained after applying the Reidemeister move once.
Now consider Figure \ref{pic02}. In Figure \ref{pic02}(a), the value $w_{c}^{(0)}$ is uniquely determined by the equation
$\exp(w_a\frac{\partial T(W)}{\partial w_a})=1$. In Figure \ref{pic02}(b), the equation $\exp(w_e\frac{\partial T(W)}{\partial w_e})=1$
uniquely determines the value $w_{d}^{(0)}=w_{b}^{(0)}$ and $w_{e}^{(0)}$ is uniquely determined by the equation
$\exp(w_b\frac{\partial T(W)}{\partial w_b})=1$. In Figure \ref{pic02}(c), the values $w_{g}^{(0)}$ and $w_{h}^{(0)}$ are 
uniquely determined by the equations $\exp(w_d\frac{\partial T(W)}{\partial w_d})=1$ and $\exp(w_a\frac{\partial T(W)}{\partial w_a})=1$, respectively.

\end{proof}

{
Now we introduce {\it the local orientation change} and show how the arc-color changes. From (\ref{matrixcc}), we obtain
\begin{equation*}
\pm i \left(\begin{array}{cc}\alpha &\beta\end{array}\right)  
  \longleftrightarrow\left(\begin{array}{cc}1-\alpha\beta & -\beta^2 \\ \alpha^2& 1+\alpha\beta\end{array}\right)
  =\left(\begin{array}{cc}1+\alpha\beta & \beta^2 \\ -\alpha^2& 1-\alpha\beta\end{array}\right)^{-1}
\end{equation*}
and $s*^{-1}(\pm i a_k)=s*a_k.$ Therefore we define the local orientation change as in Figure \ref{orichang}.
Note that the arc-color of the reversed orientated arc changes to $\pm i a_k$, but the region-colors are still well-defined.
Therefore, the invariance of the region-colors shows the invariance of the solution under the local orientation change.

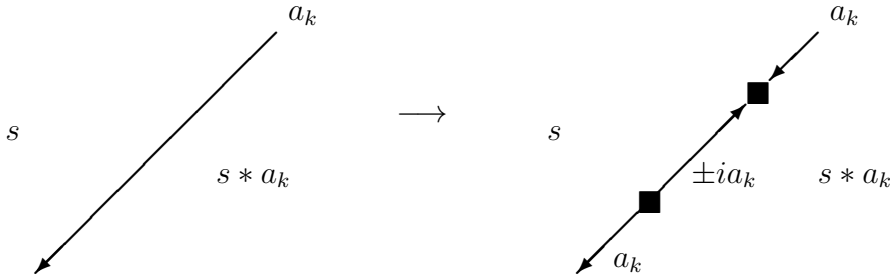
\begin{figure}[h]
\centering  \setlength{\unitlength}{0.8cm}\thicklines
\begin{picture}(20,5)  
    \put(6,4){\vector(-1,-1){4}}
    \put(1.5,2.2){$s$}
    \put(5,1.5){$s*a_k$}
    \put(6.2,4.2){$a_k$}
    \put(8,2.5){$\longrightarrow$}
    \put(15,4){\vector(-1,-1){0.8}}
    \put(12,1){\vector(1,1){1.8}}
    \put(12,1){\vector(-1,-1){1}}
    \put(12,1){$\blacksquare$}
    \put(13.8,2.8){$\blacksquare$}
    \put(10.5,2.2){$s$}
    \put(15,1.5){$s*a_k$}
    \put(15.2,4.2){$a_k$}
    \put(11.6,0.1){$a_k$}
    \put(12.9,1.5){$\pm i a_k$}
  \end{picture}
  \caption{Local orientation change}\label{orichang}
\end{figure}
}

\begin{figure}
\centering
\subfigure[First moves]{\includegraphics[scale=1.2]{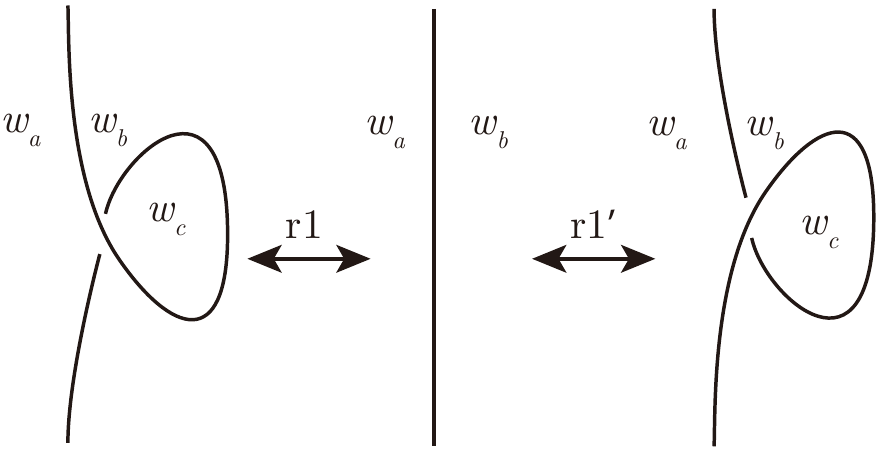}}\hspace{0.5cm}
\subfigure[Second moves]{\includegraphics[scale=1.2]{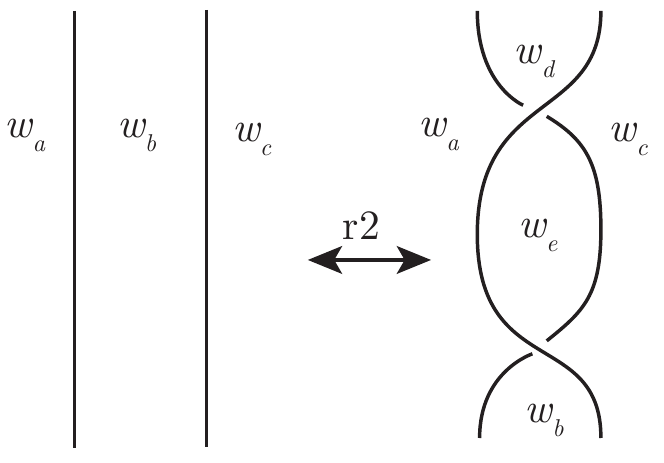}}\hspace{0.5cm}
\subfigure[Third move]{\includegraphics[scale=1.2]{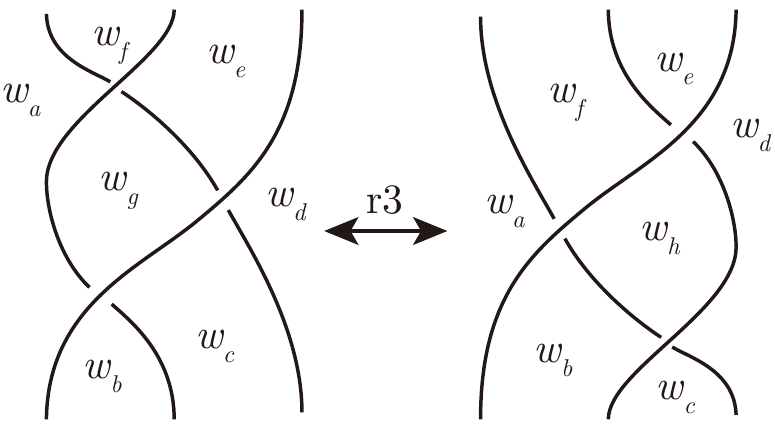}}
 \caption{Un-oriented Reidemeister moves}\label{pic15}
\end{figure}

\begin{proposition}\label{unR}
Consider the un-oriented Reidemeister moves in Figure \ref{pic15} and a link diagram $D$ containing one of Figure \ref{pic15}. 
Let $(\ldots,w_a^{(0)},w_b^{(0)},\ldots)$ be the solution of the hyperbolicity equations 
obtained by the shadow-coloring induced by $\rho$.
Then, for Figure \ref{pic15}(a), the values of the variables satisfy
\begin{equation}\label{eq20}
w_c^{(0)}=2w_b^{(0)}-w_a^{(0)},
\end{equation}
and, for Figure \ref{pic15}(c), the values satisfy
\begin{equation}\label{eq21}
w_d^{(0)}w_g^{(0)}-w_c^{(0)}w_e^{(0)}=w_a^{(0)}w_h^{(0)}-w_b^{(0)}w_f^{(0)}.
\end{equation}
\end{proposition}

\begin{proof}
{
By applying the local orientation change whenever it is necessary, we can assign the local orientation of Figure \ref{pic02}
to the un-oriented diagrams in Figure \ref{pic15}. 
Then, from the oriented Reidemeister transformations, we obtain the relations (\ref{eq20}) and (\ref{eq21}).
The solution of the hyperbolicity equations is invariant under the local orientation change, 
so the relations are independent with the choice of the orientations.

}

\end{proof}

As a result of Proposition \ref{unR}, we define the {\it un-oriented Reidemeister 1st and 3rd transformations of solutions}
by the same formulas of the oriented version in Definition \ref{def1}. On the other hand, 
the formula of the oriented Reidemeister 2nd move in Definition \ref{def1} needed an explicit potential function, which depends on the orientation.
However, we can formulate the Reidemeister 2nd move without orientations using the following lemma. At first, we define
the weight of the corner as in Figure \ref{pic16}.

\begin{figure}[h]\centering
\setlength{\unitlength}{0.6cm}
\subfigure[]{
  \begin{picture}(20,4)\thicklines
    \put(6,4){\line(-1,-1){4}}
    \put(2,4){\line(1,-1){1.8}}
    \put(4.2,1.8){\line(1,-1){1.8}}
    \put(4,2){\arc(-0.6,-0.6){90}}
    \put(3.5,0){$w_a$}
    \put(5.5,2){$w_b$}
    \put(3.5,3.5){$w_c$}
    \put(1.5,2){$w_d$}
    \put(7,2){$\longrightarrow$}
    \put(9,2){$\displaystyle x_a^j=\frac{(w_b-w_a)(w_d-w_a)}{w_b w_d-w_c w_a}$}
    \put(4,0.6){$x_a^j$}
    \put(4.5,2){$j$}
  \end{picture}}
  \subfigure[]{
  \begin{picture}(20,4)\thicklines
    \put(2,4){\line(1,-1){4}}
    \put(2,0){\line(1,1){1.8}}
    \put(6,4){\line(-1,-1){1.8}}
    \put(4,2){\arc(-0.6,-0.6){90}}
    \put(3.5,0){$w_a$}
    \put(5.5,2){$w_b$}
    \put(3.5,3.5){$w_c$}
    \put(1.5,2){$w_d$}
    \put(7,2){$\longrightarrow$}
    \put(9,2){$\displaystyle x_a^j=\frac{w_b w_d-w_c w_a}{(w_b-w_a)(w_d-w_a)}$}
    \put(4,0.6){$x_a^j$}
    \put(4.5,2){$j$}
      \end{picture}}  \caption{Weight $x_a^j$ assigned to the corner of the crossing $j$}\label{pic16}
\end{figure}

\begin{lemma}\label{lem53}
For an oriented link diagram $D$, let $W$ be the potential function of $D$ and choose a region R assigned with variable $w_k$.
Then,
$$\exp(w_k\frac{\partial W}{\partial w_k})=\prod_{j}x_k^j,$$
where $j$ is over all the crossings adjacent to the region R.
\end{lemma}

\begin{proof}
The proof can be easily obtained by direct calculations. In the case of Figure \ref{pic01}(a),
\begin{align*}
&\exp(w_a\frac{\partial W^j}{\partial w_a})=\frac{(w_b-w_a)(w_c-w_a)}{w_bw_d-w_cw_a}=x_a^j,
&\exp(w_b\frac{\partial W^j}{\partial w_b})=\frac{w_aw_c-w_dw_b}{(w_a-w_b)(w_c-w_b)}=x_b^j,\\
&\exp(w_c\frac{\partial W^j}{\partial w_c})=\frac{(w_b-w_c)(w_d-w_c)}{w_bw_d-w_aw_c}=x_c^j,
&\exp(w_d\frac{\partial W^j}{\partial w_d})=\frac{w_aw_c-w_bw_d}{(w_a-w_d)(w_c-w_d)}=x_d^j,
\end{align*}
and the case of Figure \ref{pic01}(b) can be obtained from Lemma \ref{lem40}.
The main equation is obtained by
$$\exp(w_k\frac{\partial W}{\partial w_k})=\exp(w_k\frac{\partial (\sum_j W^j)}{\partial w_k})
=\prod_{j}\exp(w_k\frac{\partial W^j}{\partial w_k})=\prod_{j}x_k^j,$$
where $j$ is over all the crossings adjacent to the region R. (Note that if $j$ is not adjacent to the region R, 
then $w_k\frac{\partial W^j}{\partial w_k}=0$.)

\end{proof}

The weight does not depends on the orientation, so we can describe the equation $\exp(w_b\frac{\partial W}{\partial w_b})=1$ 
in the right-hand side of Figure \ref{pic15}(b) without orientation.
This equation determines the value $w_e^{(0)}$ uniquely and it defines {\it the un-oriented Reidemeister 2nd move of the solution}.

Now consider a link diagram $D$ and its mirror image\footnote{The mirror image is defined as follows:
assume the link is in the $(x,y,z)$-space and the diagram $D$ is obtained by the projection along $z$-axis to the $(x,y)$-plane. Then $\overline{D}$
is the mirror image of $D$ by the reflection on the $(y,z)$-plane. For an example, see
Figure \ref{pic15}(c) and Figure \ref{pic17}.} $\overline{D}$. 
When the variables $w_a,w_b,\ldots$
are assigned to $D$, we always assume the same variables are assigned to the same mirrored regions. 
Let $W(w_a,w_b,\ldots)$ be the potential function of $D$. Then, from the definition of the potential functions
in Figure \ref{pic01}, the potential function of $\overline{D}$ becomes $-W(w_a,w_b,\ldots)$.
(Note that we are using here the invariance of the potential functions in Figure \ref{pic01} under $w_b\leftrightarrow w_d$.)
This suggests that the hyperbolicity equations in $\mathcal{I}$ are invariant under the mirroring. 

\begin{lemma}\label{R3mirror}
Consider the diagrams in Figure \ref{pic17} and let
$(\ldots,w_a^{(0)},\ldots,w_f^{(0)},w_g^{(0)},\ldots)$ and 
$(\ldots,w_a^{(0)},\ldots,w_f^{(0)},w_h^{(0)},\ldots)$
be the solutions of the hyperbolicity equations obtained by the shadow-coloring induced by $\rho$.
The the values satisfy the equation (\ref{eq21}).
\end{lemma}

\begin{proof} For Figure \ref{pic15}(c), the relation (\ref{eq21}) holds. 
The hyperbolicity equations are invariant under the mirroring,
{so a solution on the diagram $D$ is also a solution on $\overline{D}$. 
Therefore, a pair of solutions related by the relation (\ref{eq21}) on Figure \ref{pic15}(c) is 
also a pair of solution on Figure \ref{pic17} related by the same relation.
From the uniqueness of the solution in Lemma \ref{uniquelemma}, this is the only relation 
induced by a shadow-coloring.}

\end{proof}

\begin{figure}
\centering
\includegraphics[scale=1.2]{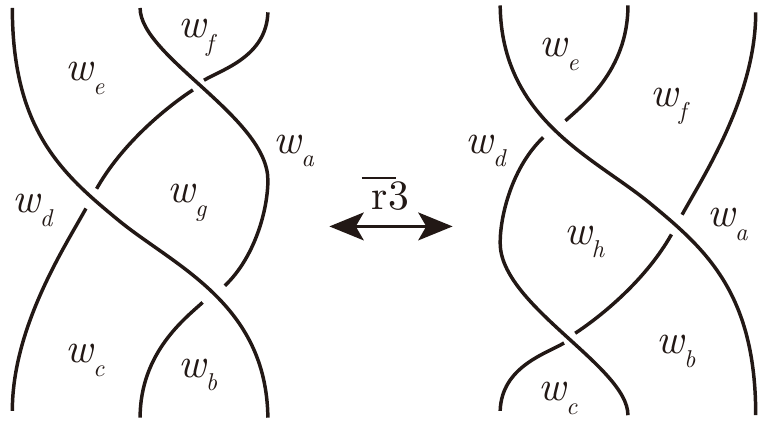}
 \caption{Mirror image of Figure \ref{pic15}(c)}\label{pic17}
\end{figure}

\begin{proposition} 
Consider a diagram $D'$ is obtained from $D$ by applying one un-oriented Reidemeister move,
assume variables $w_1,\ldots,w_n$ are assigned to regions of $D$
and $w_{n+1}$ is assigned to a newly appeared region of $D'$. Then Lemma \ref{uniquelemma} showed that
the value $w_{n+1}^{(0)}$ is uniquely determined from $w_1^{(0)},\ldots, w_n^{(0)}$ by a certain equation.
If we consider the mirror image $\overline{D}'$ obtained from $\overline{D}$, then the value $w_{n+1}^{(0)}$
of the mirror image is uniquely determined by the same equation.
\end{proposition}

\begin{proof} The mirror images of the un-oriented Reidemeister 1st and 2nd moves are just $\pi$-rotations of the original moves,
so the same equation holds for each cases. The mirror image of the un-oriented Reidemeister 3rd move 
was already proved in Lemma \ref{R3mirror}.

\end{proof}

\begin{exa}\label{exa1}
Consider Figure \ref{pic17} and let
$(\ldots,w_a^{(0)},\ldots,w_f^{(0)},w_g^{(0)},\ldots)$ and 
$(\ldots,w_a^{(0)},\ldots,w_f^{(0)},w_h^{(0)},\ldots)$
be the solutions of the hyperbolicity equations of each diagrams obtained by the shadow-colorings induced by $\rho$.
The the variables satisfy
\begin{equation*}
w_d^{(0)}w_g^{(0)}-w_c^{(0)}w_e^{(0)}=w_a^{(0)}w_h^{(0)}-w_b^{(0)}w_f^{(0)},
\end{equation*}
which is the same equation with (\ref{eq21}).
\end{exa}

\section{Examples}\label{sec6}

\begin{exa}\label{exa2}
Consider the twisting move in Figure \ref{pic18} and let
$(\ldots,w_a^{(0)},w_b^{(0)},w_c^{(0)},w_d^{(0)},w_e^{(0)},\ldots)$ and 
$(\ldots,w_a^{(0)},w_c^{(0)},w_d^{(0)},w_e^{(0)},w_f^{(0)},w_g^{(0)},\ldots)$
be the solutions of the hyperbolicity equations of each diagrams obtained by the shadow-colorings induced by $\rho$.

    \begin{figure}[h]\centering
    \includegraphics[scale=0.8]{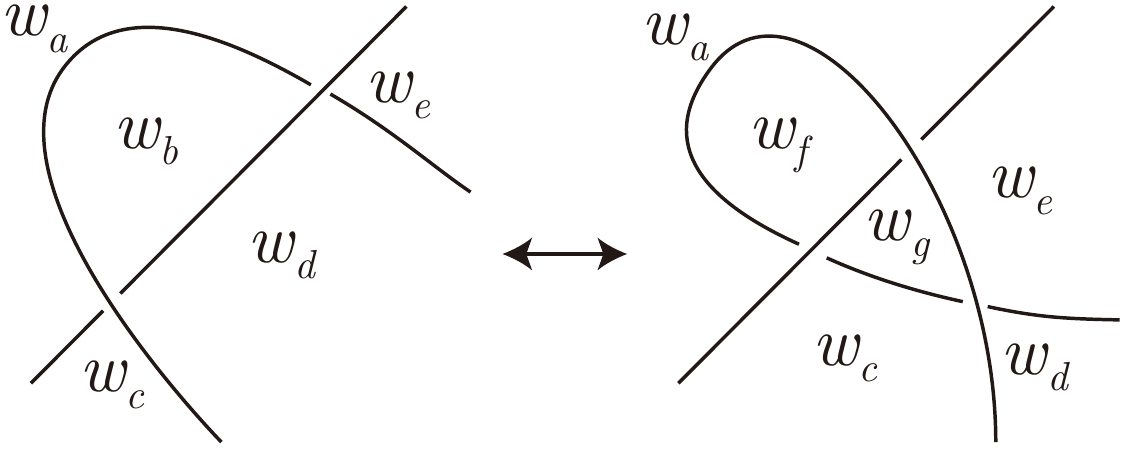}
    \caption{Twisting move}\label{pic18}
   \end{figure}

If we apply the twisting move from the left to the right, then $w_f^{(0)}$ and $w_g^{(0)}$ are uniquely determined by
$$w_f^{(0)}=2w_a^{(0)}-w_b^{(0)}~\text{ and }~w_d^{(0)} w_g^{(0)}-w_c^{(0)} w_e^{(0)}=w_f^{(0)} w_b^{(0)}-(w_a^{(0)})^2,$$
and if we apply it from the right to the left, then $w_b^{(0)}$ is uniquely determined by
$$w_b^{(0)}=2w_a^{(0)}-w_f^{(0)}.$$
\end{exa}

\begin{proof} By adding a kink from the left-hand side of Figure \ref{pic18}, we obtain Figure \ref{pic19} and the equation
$$w_f^{(0)}=2w_a^{(0)}-w_b^{(0)}.$$

    \begin{figure}[h]\centering
    \includegraphics[scale=0.8]{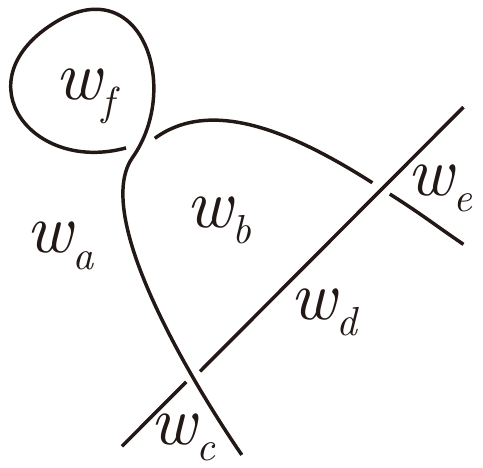}
    \caption{Adding a kink}\label{pic19}
   \end{figure}
   
Applying Example \ref{exa1} to Figure \ref{pic19}, we obtain another equation
$$w_d^{(0)} w_g^{(0)}-w_c^{(0)} w_e^{(0)}=w_f^{(0)} w_b^{(0)}-(w_a^{(0)})^2.$$
\end{proof}

Now we will show the changes of the solution from one figure-eight knot diagram to its mirror image.
At first, consider Figure \ref{example1}.

\begin{figure}[h]\centering
\includegraphics[scale=0.5]{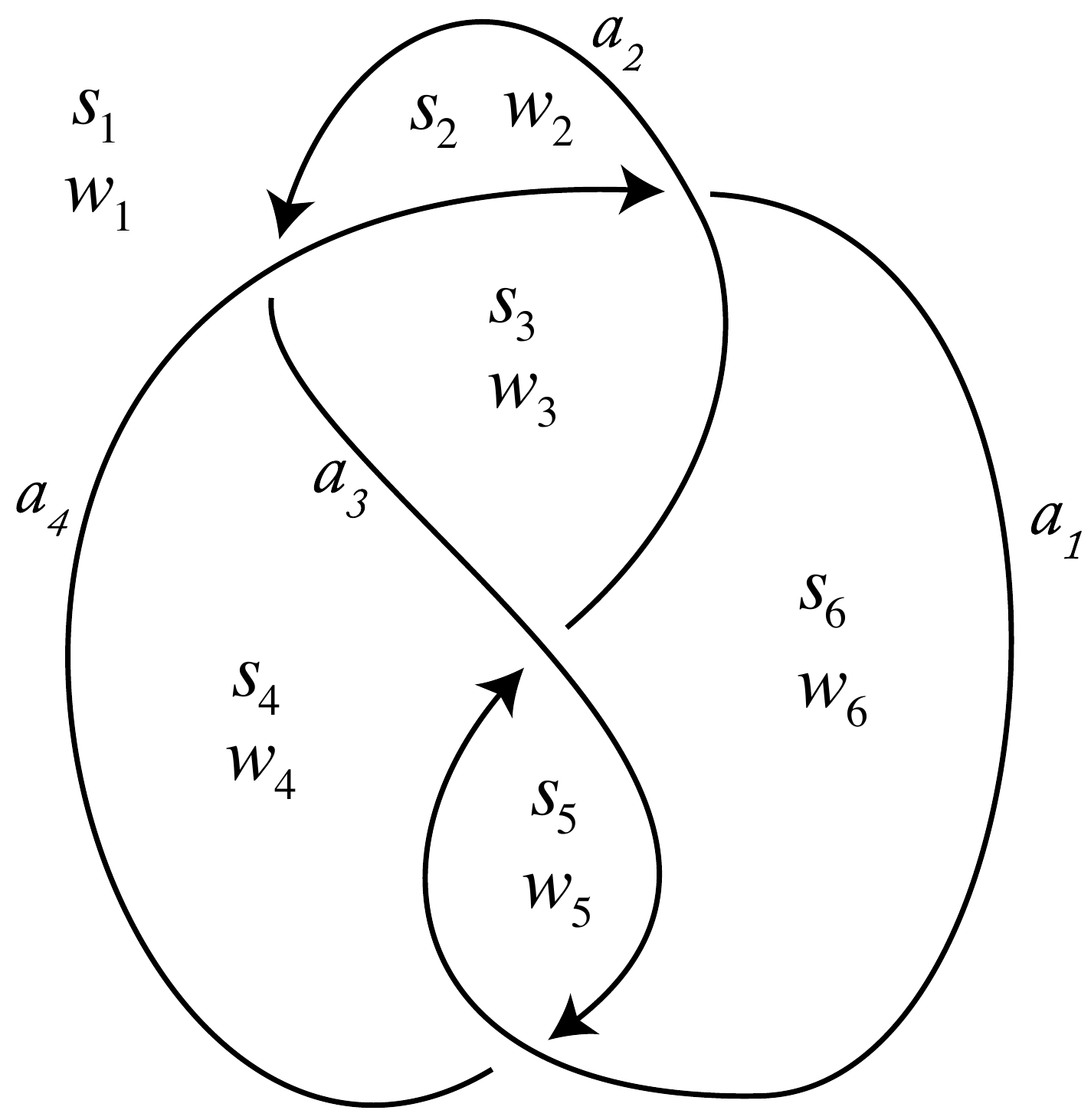}
\caption{Figure-eight knot $4_1$ with parameters}\label{example1}
\end{figure}

Let $\rho:\pi_1(4_1)\rightarrow{\rm PSL}(2,\mathbb{C})$ be the boundary-parabolic representation defined by
the arc-colors
$$a_1=\left(\begin{array}{cc}0 &t\end{array}\right),~a_2=\left(\begin{array}{cc}1 &0\end{array}\right),~
a_3=\left(\begin{array}{cc}-t &1+t\end{array}\right),~a_4=\left(\begin{array}{cc}-t &t\end{array}\right),$$
where $t$ is a solution of $t^2+t+1=0$, and let one region-color $s_1=\left(\begin{array}{cc}1 &1\end{array}\right)$.
Then the other region-colors become
\begin{align*}
s_2=\left(\begin{array}{cc}0 &1\end{array}\right),~
s_3=\left(\begin{array}{cc}-t-1 &t+2\end{array}\right),~s_4=\left(\begin{array}{cc}-2t-1 &2t+3\end{array}\right),\\
s_5=\left(\begin{array}{cc}-2t-1 &t+4\end{array}\right),~s_6=\left(\begin{array}{cc}1 &t+2\end{array}\right).
\end{align*}

The potential function $W(w_1,\ldots,w_6)$ of Figure \ref{example1} is
\begin{align*}
W=\left\{\li(\frac{w_1}{w_2})+\li(\frac{w_1}{w_4})-\li(\frac{w_1 w_3}{w_2 w_4})
-\li(\frac{w_2}{w_3})-\li(\frac{w_4}{w_3})+\frac{\pi^2}{6}-\log\frac{w_2}{w_3}\log\frac{w_4}{w_3}\right\}\\
+\left\{\li(\frac{w_3}{w_2})+\li(\frac{w_3}{w_6})-\li(\frac{w_1 w_3}{w_2 w_6})
-\li(\frac{w_2}{w_1})-\li(\frac{w_6}{w_1})+\frac{\pi^2}{6}-\log\frac{w_2}{w_1}\log\frac{w_6}{w_1}\right\}\\
+\left\{-\li(\frac{w_4}{w_3})-\li(\frac{w_4}{w_5})+\li(\frac{w_4 w_6}{w_3 w_5})
+\li(\frac{w_3}{w_6})+\li(\frac{w_5}{w_6})-\frac{\pi^2}{6}+\log\frac{w_3}{w_6}\log\frac{w_5}{w_6}\right\}\\
+\left\{-\li(\frac{w_6}{w_1})-\li(\frac{w_6}{w_5})+\li(\frac{w_4 w_6}{w_1 w_5})
+\li(\frac{w_1}{w_4})+\li(\frac{w_5}{w_4})-\frac{\pi^2}{6}+\log\frac{w_1}{w_4}\log\frac{w_5}{w_4}\right\}.
\end{align*} 

By putting $p=\left(\begin{array}{cc}2 &1\end{array}\right)$, we obtain
\begin{align}
w_1^{(0)}=\det(p,s_1)=1,~w_2^{(0)}=\det(p,s_2)=2,~w_3^{(0)}=\det(p,s_3)=3t+5,\label{sol1}\\
w_4^{(0)}=\det(p,s_4)=6t+7,~w_5^{(0)}=\det(p,s_5)=4t+9,~w_6^{(0)}=\det(p,s_6)=2t+3,\nonumber
\end{align}
and $(w_1^{(0)},\ldots,w_6^{(0)})$ becomes a solution of 
$\mathcal{I}=\{\exp(w_k\frac{\partial W}{\partial w_k})=1~|~k=1,\ldots,6\}$.
Furthermore, we obtain
\begin{equation*}
W_0(w_1^{(0)},\ldots,w_6^{(0)})\equiv i(\vol(\rho)+i\,\cs(\rho))\modulo,
\end{equation*}
and numerical calculation verifies it by
\begin{equation*}
W_0(w_1^{(0)},\ldots,w_6^{(0)})=
      \left\{\begin{array}{ll}i(2.0299...+0\,i)=i(\vol(4_1)+i\,\cs(4_1))&\text{ if }t=\frac{-1-\sqrt{3} \,i}{2}, \\
                i(-2.0299...+0\,i)=i(-\vol(4_1)+i\,\cs(4_1))&\text{ if }t=\frac{-1+\sqrt{3}\,i}{2}. \end{array}\right.
\end{equation*}

Note that the above example was already appeared in Section 3.1.\;of \cite{Cho14c}. 
From Theorem \ref{thm1}, we can easily specify the discrete faithful representation by
Figure \ref{example1} together with the solution (\ref{sol1}). (The explicit construction of the representation
can be done by applying Yoshida's construction of \cite{Tillmann13} to the five-term triangulation defined in Figure \ref{fig7}.)

Now we will apply (un-oriented) Reidemeister moves to the solution in (\ref{sol1}).
Consider the changes of the figure-eight knot diagrams in Figure \ref{mirror4_1}.

\begin{figure}[ht]
\centering
\subfigure[]{\includegraphics[scale=0.65]{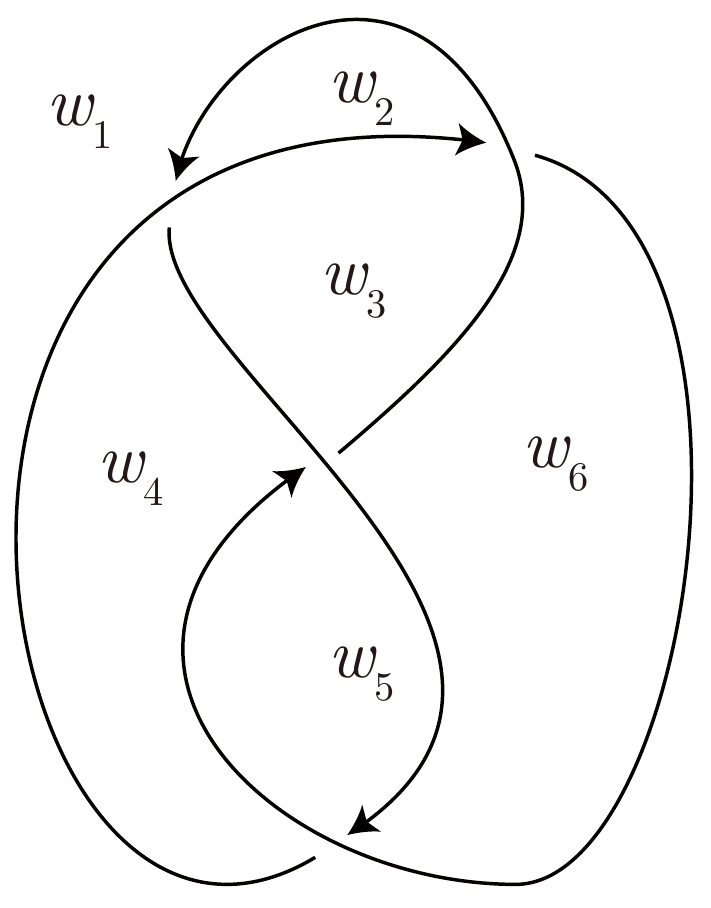}}
\subfigure[]{\includegraphics[scale=0.65]{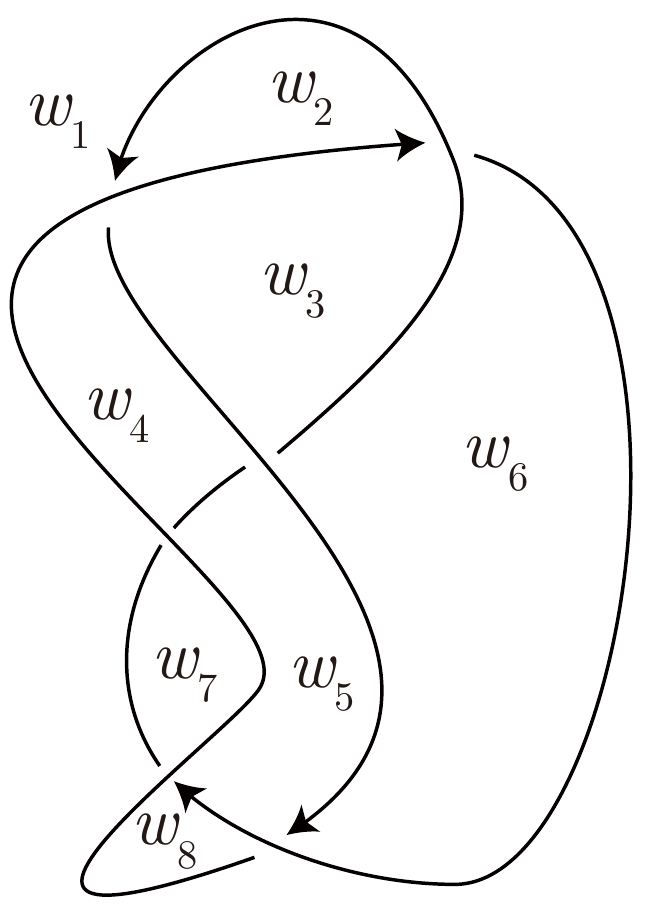}}
\subfigure[]{\includegraphics[scale=0.65]{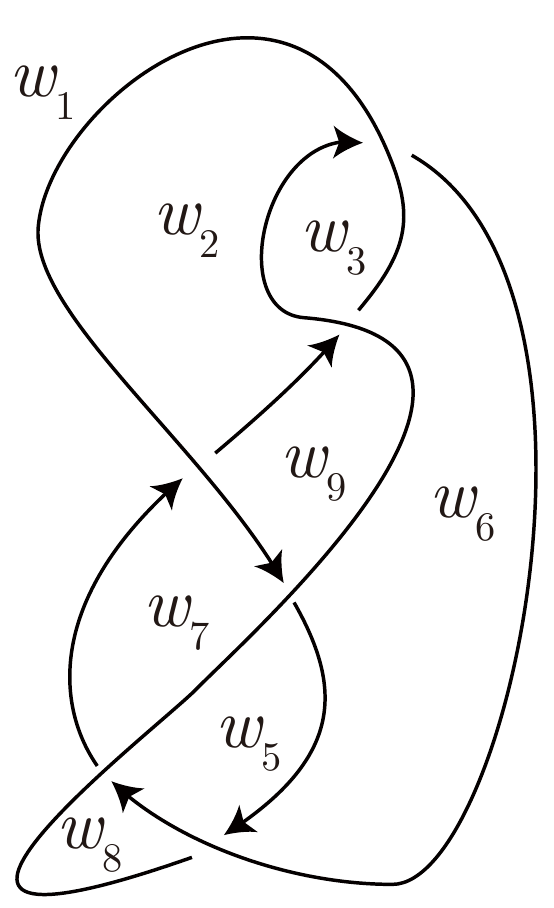}}
\subfigure[]{\includegraphics[scale=0.65]{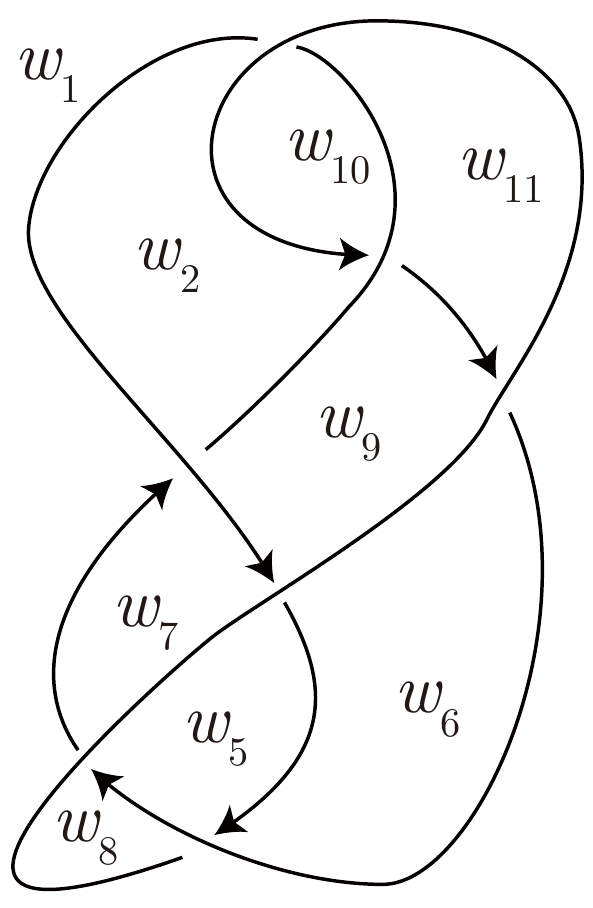}}
\subfigure[]{\includegraphics[scale=0.65]{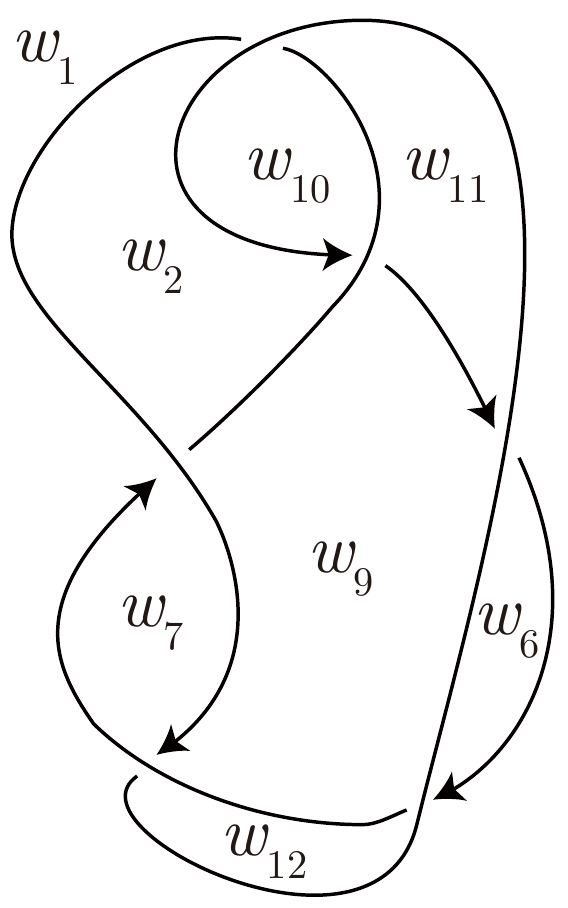}}
\subfigure[]{\includegraphics[scale=0.65]{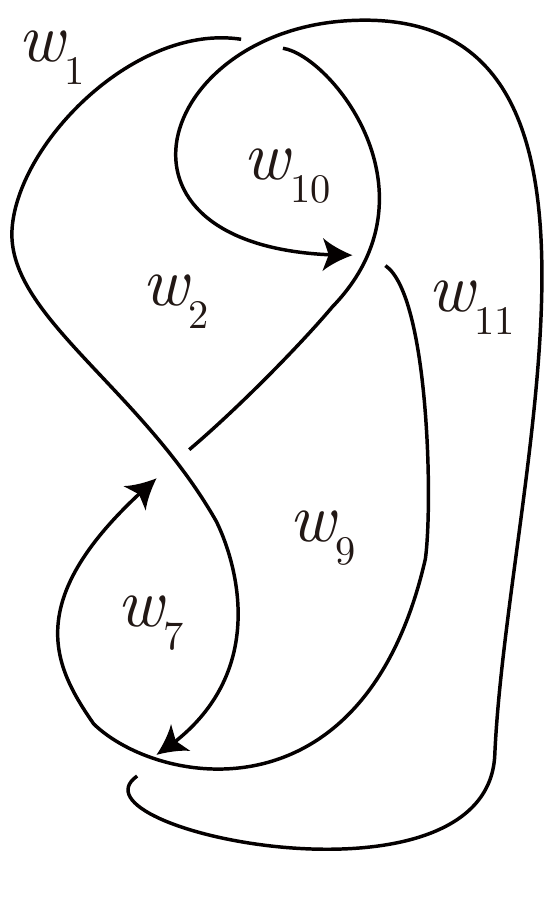}}
\subfigure[]{\includegraphics[scale=0.65]{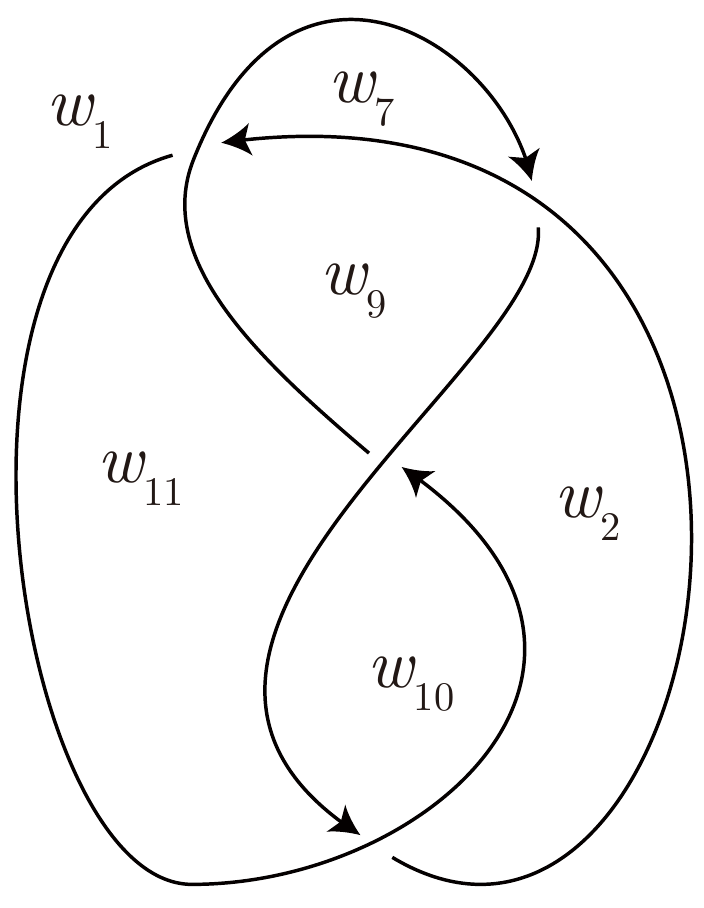}}
 \caption{Changing the figure-eight knot diagram to its mirror image}\label{mirror4_1}
\end{figure}

Note that Figure \ref{mirror4_1}(c) is obtained by the Reidemeister 3rd move,
Figures \ref{mirror4_1}(d)-(e) are obtained by the twist move defined in Figure \ref{pic18}
and Figure \ref{mirror4_1}(g) is obtained by the rotation.
Then the values of the variables are determined by
\begin{align*}
w_7^{(0)}&=\frac{w_1^{(0)}w_5^{(0)}(w_3^{(0)}-w_4^{(0)})^2-(w_1^{(0)}w_3^{(0)}-w_2^{(0)}w_4^{(0)})(w_3^{(0)}w_5^{(0)}-w_4^{(0)}w_6^{(0)})}{w_4^{(0)}(w_3^{(0)}-w_4^{(0)})^2}\\
&=\frac{6t^2+37t+36}{(3t+2)^2}=-5t-3,\\
w_8^{(0)}&=w_4^{(0)}=6t+7,\\
w_9^{(0)}&=\frac{w_4^{(0)}w_7^{(0)}-w_1^{(0)}w_5^{(0)}+w_2^{(0)}w_6^{(0)}}{w_3^{(0)}}=\frac{-30t^2-53t-24}{3t+5}=-7t-3,\\
w_{10}^{(0)}&=2w_2^{(0)}-w_3^{(0)}=-3t-1,\\
w_{11}^{(0)}&=\frac{w_3^{(0)}w_{10}^{(0)}-(w_2^{(0)})^2+w_1^{(0)}w_9^{(0)}}{w_6^{(0)}}=\frac{-9t^2-25t-12}{2t+3}=-6t-5,\\
w_{12}^{(0)}&=2w_1^{(0)}-w_8^{(0)}=-6t-5=w_{11}^{(0)}.
\end{align*}
(Here, $w_7^{(0)}$ is calculated by the partial derivative of the potential function with respect to $w_4$.)

The potential function of Figure \ref{mirror4_1}(g) becomes $-W(w_1,w_7,w_9,w_{2},w_{10},w_{11})$ and 
the following numerical calculation
\begin{align*}
-W_0(w_1^{(0)},w_7^{(0)},w_9^{(0)},w_{2}^{(0)},w_{10}^{(0)},w_{11}^{(0)})&=
      \left\{\begin{array}{ll}i(2.0299...+0\,i)&\text{ if }t=\frac{-1-\sqrt{3} \,i}{2} \\
                i(-2.0299...+0\,i)&\text{ if }t=\frac{-1+\sqrt{3}\,i}{2} \end{array}\right.\\
&=W_0(w_1^{(0)},w_2^{(0)},w_3^{(0)},w_4^{(0)},w_5^{(0)},w_6^{(0)})
\end{align*}
confirms Theorem \ref{thm1}.

\appendix
\section{Shadow-coloring induced by a solution}\label{app}

In \cite{Cho14c} and this article, we always start from a given boundary-parabolic representation $\rho:\pi_1(L)\rightarrow{\rm PSL}(2,\mathbb{C})$
and construct a solution $(w_1^{(0)},\ldots,w_n^{(0)})$ of the hyperbolicity equations $\mathcal{I}$ using (\ref{main}). 
In other words, for any representation $\rho$, we can always construct a solution that induces $\rho$. Therefore natural question arises that 
whether any {essential} solution of $\mathcal{I}$ can be constructed by the formula (\ref{main}) of certain shadow-coloring.
This question is important because, if it is true, then any essential solution of $\mathcal{I}$ is governed by the Reidemeister transformations.
Furthermore, we can characterize the solutions of $\mathcal{I}$ by the choices of certain shadow-coloring. 

\begin{theorem} For any {essential} solution 
$\bold{w}^{(0)}=(w_1^{(0)},\ldots,w_n^{(0)})$ of $\mathcal{I}$, 
there exist an arc-coloring $\{a_1,\ldots,a_r\}$,
a region coloring $\{s_1,\ldots,s_n\}$ and an element $p\in\mathbb{C}^2\backslash\{0\}$ satisfying 
\begin{equation*}
w_k^{(0)}=\det(p,s_k),
\end{equation*}
for $k=1,\ldots,n$.
\end{theorem}

\begin{proof}
From the discussion above of Proposition \ref{pro21}, the solution $\bold{w}^{(0)}$ induces the boundary-parabolic representation $\rho$.
After fixing an oriented diagram $D$ of the link $L$, the representation $\rho$ induces unique arc-coloring $\{a_1,\ldots,a_r\}$ up to conjugation.
By using proper conjugation, we may assume 
$$\infty\notin\{h(a_1),\ldots,h(a_r)\}.$$ Then we define $p=\left(\begin{array}{cc}1 &0\end{array}\right)$.

The main idea of this proof is to show the following region-coloring
$$s_k=w_k^{(0)}\left(\begin{array}{cc}h(s_k) &1\end{array}\right)$$
for $k=1,\ldots,n$, is what we want. To prove rigorously, assume the regions with $s_1$ and $s_2$ are adjacent, as in Figure \ref{fig22}.

\begin{figure}[h]
\centering  \setlength{\unitlength}{0.6cm}\thicklines
\begin{picture}(6,5)  
    \put(6,4){\vector(-1,-1){4}}
    \put(1.5,2.2){$s_1$}
    \put(5,1.5){$s_2$}
    \put(6.2,4.2){$a_l=\left(\begin{array}{cc}\alpha_l &\beta_l\end{array}\right)$}
  \end{picture}
  \caption{Defining the region-color $s_1$}\label{fig22}
\end{figure}
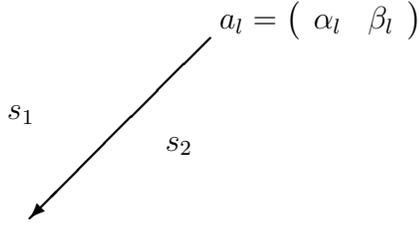

Define $s_1$ by
$$s_1=w_1^{(0)}\left(\begin{array}{cc}\left(\alpha_l \beta_l-1+\frac{w_2^{(0)}}{w_1^{(0)}}\right)/\beta_l^2 & 1 \end{array}\right).$$
Then $h(s_1)=\left(\alpha_l \beta_l-1+\frac{w_2^{(0)}}{w_1^{(0)}}\right)/\beta_l^2$. 
We can decide $h(s_2),\ldots,h(s_n)$ from $h(s_1)$ using the arc-coloring $\{a_1,\ldots,a_r\}$ together with (\ref{app1}), 
and we denote the region-colorings by
\begin{equation}\label{app2}
s_k=x_k\left(\begin{array}{cc}h(s_k) &1\end{array}\right)
\end{equation}
for $k=1,\ldots,n$. (Therefore, $x_1=w_1^{(0)}$ and $x_2,\ldots,x_n$ are uniquely determined.)
Then, from the second row of the following matrices
$$s_2=x_2\left(\begin{array}{cc}h(s_2) &1 \end{array}\right)=s_1*\alpha_l
=x_1\left(\begin{array}{cc}h(s_1) &1 \end{array}\right) \left(\begin{array}{cc}1+\alpha_l\beta_l &\beta_l^2   \\-\alpha_l^2& 1-\alpha_l\beta_l\end{array}\right),$$
we obtain 
$$x_2=x_1(\beta_l^2 h(s_1)+1-\alpha_l\beta_l)=x_1\frac{w_2^{(0)}}{w_1^{(0)}}=w_2^{(0)}.$$

Now let the developing map induced by the solution $\bold{w}^{(0)}$ be $D_1$, and the one induced by
the shadow-coloring $\{a_1,\ldots,a_r,s_1,\ldots,s_n,p\}$ together with (\ref{app2}) be $D_2$.
(The definition of the developing map we are using here is {\it Definition 4.10} from Section 4 of \cite{Zickert09}.)
Note that $D_1$ can be constructed by gluing tetrahedra with the shape parameters determined by $\bold{w}^{(0)}$, 
and $D_2$ is constructed explicitly at Figure \ref{fig06}.
(The developing map $D_2$ satisfies the condition (4.2) in {\it Proof} of THEOREM 4.11 of \cite{Zickert09}.) 
To make $D_1=D_2$,
consider the five-term triangulation defined in Section \ref{sec2} and see Figure \ref{fig23}.
\begin{figure}[H]
\centering
  \subfigure[Positive crossing $j$]
  {\begin{picture}(6,4)  
  \setlength{\unitlength}{0.8cm}\thicklines
        \put(4,2){\arc[5](-1,1){180}}
    \put(6,4){\vector(-1,-1){4}}
    \put(2,4){\line(1,-1){1.8}}
    \put(4.2,1.8){\line(1,-1){1.8}}
    \put(3.7,0.2){$w_2$}
    \put(1.7,2){$w_1$}
    \put(2.2,0.9){${\rm A}_j$}
    \put(5.3,0.9){${\rm B}_j$}
    \put(2.2,2.9){${\rm D}_j$}
    \put(4.3,2){$j$}
  \end{picture}}
  \subfigure[Negative crossing $j$]
  {\begin{picture}(6,4)
  \setlength{\unitlength}{0.8cm}\thicklines
        \put(4,2){\arc[5](-1,1){180}}
    \put(2,4){\line(1,-1){4}}
   \put(6,4){\line(-1,-1){1.8}}
   \put(3.8,1.8){\vector(-1,-1){1.8}}
    \put(3.7,0.2){$w_2$}
    \put(1.7,2){$w_1$}
    \put(2.2,0.9){${\rm D}_j$}
    \put(5.3,0.9){${\rm A}_j$}
    \put(2.2,2.9){${\rm C}_j$}
    \put(4.3,2){$j$}
  \end{picture}}
  \caption{Figure \ref{fig22} together with the crossing $j$}\label{fig23}
\end{figure}
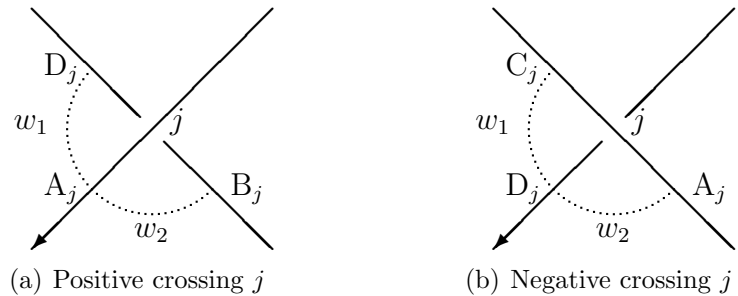

We also consider the octahedron ${\rm A}_j{\rm B}_j{\rm C}_j{\rm D}_j{\rm E}_j{\rm F}_j$ at the crossing $j$ as in Figure \ref{fig7}. 
(Note that Figure \ref{fig23} is the left-hand side of Figure \ref{fig7}.)
The property $x_1=w_1^{(0)}$ and $x_2=w_2^{(0)}$ imply that, in case of Figures \ref{fig7}(a) and \ref{fig23}(a), 
the shape parameter of $\Delta{\rm D}_j{\rm B}_j{\rm F}_j{\rm A}_j$ assigned to ${\rm F}_j{\rm A}_j$ is $\frac{w_1^{(0)}}{w_2^{(0)}}$ and
the shape parameter of $\Delta{\rm D}_j{\rm E}_j{\rm A}_j{\rm C}_j$ assigned to ${\rm D}_j{\rm E}_j$ is $\frac{w_2^{(0)}}{w_1^{(0)}}$.
These shape parameters coincide with the shape parameters determined by the solution $\bold{w}^{(0)}$.
Hence, for the lifts $\widetilde{{\rm A}}_j$, $\widetilde{{\rm B}}_j$, ..., $\widetilde{{\rm F}}_j$ of the vertices,
we can put 
\begin{equation}\label{app3}
D_1(\widetilde{{\rm D}}_j)=D_2(\widetilde{{\rm D}}_j),~D_1(\widetilde{{\rm B}}_j)=D_2(\widetilde{{\rm B}}_j),~
D_1(\widetilde{{\rm F}}_j)=D_2(\widetilde{{\rm F}}_j), D_1(\widetilde{{\rm A}}_j)=D_2(\widetilde{{\rm A}}_j)
\end{equation}
in the case of Figures \ref{fig7}(a) and \ref{fig23}(a), and put 
\begin{equation}\label{app4}
D_1(\widetilde{{\rm D}}_j)=D_2(\widetilde{{\rm D}}_j),~ D_1(\widetilde{{\rm E}}_j)=D_2(\widetilde{{\rm E}}_j),~
D_1(\widetilde{{\rm A}}_j)=D_2(\widetilde{{\rm A}}_j),~ D_1(\widetilde{{\rm C}}_j)=D_2(\widetilde{{\rm C}}_j)
\end{equation}
in the case of Figures \ref{fig7}(b) and \ref{fig23}(b).

Note that the developing maps $D_1$ and $D_2$ are defined from the same representation $\rho$. Therefore,
from the uniqueness theorem of the developing map in THEOREM 4.11 of \cite{Zickert09},
$D_1$ and $D_2$ agree on the ideal points corresponding to the {\it nontrivial} ends. 
(See {\it Definition 4.3} of \cite{Zickert09} for the definitions of the nontrivial and trivial ends.)
Furthermore, the five-term triangulation we are using has two trivial ends 
and we denoted the corresponding points by $\pm\infty$ in Section \ref{sec2}.
At the octahedron in Figure \ref{fig7}, the ideal points $\widetilde{{\rm A}}_j$ and $\widetilde{{\rm C}}_j$
corresponds to $\infty$ and $\widetilde{{\rm B}}_j$ and $\widetilde{{\rm D}}_j$
corresponds to $-\infty$. From (\ref{app3}) and (\ref{app4}), the two developing maps $D_1$ and $D_2$ coincide
not only at the points corresponding to the nontrivial ends, but also at the points corresponding to the trivial ends.
Therefore, we obtain $D_1=D_2$. 

The coincidence of the two developing maps implies the shape parameters of the tetrahedra in the five-term triangulation coincide everywhere. Therefore, from (\ref{main}), we obtain
$$\frac{x_k}{x_m}=\frac{w_k^{(0)}}{w_m^{(0)}}$$
for any adjacent regions with $w_k$ and $w_m$. (Note that $\left(\frac{x_k}{x_m}\right)^{\pm}$ is the shape parameter
determined by the formula (\ref{main}) and the developing map $D_2$.) We already know $x_1=w_1^{(0)}$ and $x_2=w_2^{(0)}$,
so we obtain $x_k=w_k^{(0)}$ for $k=1,\ldots,n$, and the shadow-coloring $\{a_1,\ldots,a_r,s_1,\ldots,s_n,p\}$ is what we want.

\end{proof}

\vspace{5mm}
\begin{ack}
The first author is supported by Basic Science Research Program through the National Research Foundation of Korea funded by the Ministry of Education (NRF-2015R1C1A1A02037540). The appendix is motivated by the discussion with Christian Zickert and the authors appreciate him. {The authors also show gratitude to the anonymous reviewer 
who suggested better proof of Lemma \ref{lem} than the original.}
\end{ack}

{
\begin{flushleft}
  Busan National University of Education\\ Busan 47503, Republic of Korea\\
E-mail: dol0425@bnue.ac.kr\\
  \vspace{0.4cm}
Department of Mathematics,
Faculty of Science and Engineering,
Waseda University\\
3-4-1 Ohkumo, Shinjuku-ku, Tokyo
169-855, Japan\\
E-mail: murakami@waseda.jp\\
\end{flushleft}}

\end{document}